\documentclass[a4paper,leqno]{amsart}

\usepackage{latexsym}
\usepackage[english]{babel}
\usepackage{fancyhdr}
\usepackage[mathscr]{eucal}
\usepackage{amsmath}
\usepackage{mathrsfs}
\usepackage{amsthm}
\usepackage{amsfonts}
\usepackage{amssymb}
\usepackage{amscd}
\usepackage{bbm}
\usepackage{graphicx}
\usepackage{subcaption}
\usepackage{graphics}
\usepackage{latexsym}
\usepackage{color}

\usepackage{pifont}
\usepackage{booktabs} 

\newcommand{\ud}{\mathrm{d}}

\newcommand{\cH}{\mathcal{H}}

\newcommand{\N}{\mathbb N}

\theoremstyle{plain}
\newtheorem{theorem}{Theorem}[section]
\newtheorem{lemma}[theorem]{Lemma}

\theoremstyle{definition}

\newtheorem{remark}[theorem]{Remark}
\newtheorem{example}[theorem]{Example}

\numberwithin{equation}{section}

\begin{document}

\title[Projection and convergence for abstract linear inverse problems]
{On general convergence behaviours \\ of finite-dimensional approximants \\ for abstract linear inverse problems}

\author[No\`e Angelo Caruso]{No\`e Angelo Caruso}

\address[No\`e Angelo Caruso]{%
International School for Advanced Studies\\
via Bonomea 265, I-34136 Trieste (ITALY)\\
and Gran Sasso Science Institute\\
Viale F.~Crispi 7, I-67100 L'Aquila (ITALY)}

\email{noe.caruso@gssi.it}

\author[Alessandro Michelangeli]{Alessandro Michelangeli}
\address[Alessandro Michelangeli]{Institute for Applied Mathematics \\
and Hausdorff Center for Mathematics\\
University of Bonn\\
Endenicher Allee 60, D-53115 Bonn (GERMANY)}
\email{michelangeli@iam.uni-bonn.de}

\author[P.~Novati]{Paolo Novati}
\address[P.~Novati]{Universit\`{a} degli Studi di Trieste \\ Piazzale Europa 1 \\ I-34127 Trieste (Italy).}
\email{novati@units.it}


\begin{abstract}
In the framework of abstract linear inverse problems in infinite-dimensional Hilbert space we discuss generic convergence behaviours of approximate solutions determined by means of general projection methods, namely outside the standard assumptions of Petrov-Galerkin truncation schemes. This includes 
a discussion of the mechanisms why the error or the residual generically fail to vanish in norm, and the identification of practically plausible sufficient conditions for such indicators to be small in some weaker sense. The presentation is based on theoretical results together with a series of model examples and numerical tests.
\end{abstract}

\date{\today}

\makeatletter{\renewcommand*{\@makefnmark}{}
\footnotetext{Asymptotic Analysis (2021)}\makeatother}




\thanks{This work is partially supported by the Alexander von Humboldt foundation}

\maketitle


\section{Introduction. Inverse problems and general projection methods.}\label{intro}

The goal of this work is to present \emph{generic convergence phenomena} for approximate solutions to linear inverse problems. In particular we focus on (weak) convergence mechanisms that usually are not considered in typical (strong) approximation schemes: our long term ambition, indeed, is to systematically describe a much wider variety of weaker, yet fully informative ways to control the finite-dimensional truncations of an infinite-dimensional inverse problem and the possible asymptotic convergence of the approximated solutions to an actual solution for the original problem. At this initial stage of investigation we cannot be so comprehensive and we content ourselves to outline certain general mechanisms. As a compensation for such initial lack of systematicness, we supplement our reasonings with an amount of instructive and somewhat unexpected examples.

To begin with, let us clarify the level of generality for our discussion. It is purposely chosen so as to encompass a \emph{wider} class of solution schemes than the customary `Petrov-Galerkin' ones, as we shall discuss in a moment.
\begin{itemize}
 \item The \emph{inverse problem} is modelled abstractly as an inverse problem on Hilbert space, the latter being possibly infinite-dimensional.
 \item The \emph{projection method} implemented to produce finite-dimensional truncations of the original problem, and hence approximate solutions to it, only undergoes minimal operational assumptions making the algorithm well defined, yet allowing for singular truncated problems with no a priori guaranteed convergence theory in the strong sense.
 \end{itemize}

Let us describe the above framework in detail.

To fix the nomenclature and the notation, by an \emph{abstract linear inverse problem} one means a \emph{linear inverse problem on Hilbert space}, namely the problem, given a Hilbert space $\cH$, a linear operator $A$ acting on $\cH$, and a vector $g\in\cH$, to determine the solution(s) $f\in\cH$ to the linear equation
\begin{equation}\label{eq:gen}
 Af\;=\;g\,.
\end{equation}
We shall say that: \eqref{eq:gen} is \emph{solvable} if a solution $f$ exists, namely if $g\in\mathrm{ran}A$; \eqref{eq:gen} is \emph{well-defined} if additionally the solution $f$ is unique, i.e., if $A$ is also injective (in which case one refers to  $f$ as the `\emph{exact}' solution); \eqref{eq:gen} is \emph{well-posed} if there exists a unique solution that depends continuously (i.e., in the norm of $\cH$) on the datum $g$, equivalently, that $g\in\cH$ and $A$ has everywhere defined bounded inverse.

In applications, the linear law $A$ that associates an input $f$ to an output $g$ is prescribed by some physical model, and hence within that model such a law is exactly known. Experimental measurements produce a possibly approximate knowledge of the output $g$, from which one wants to obtain information on the input $f$.

Of course what is `exactly known' of $A$ is its domain and action as an operator acting on $\cH$. Other relevant features of $A$ might not be explicitly accessible, and only computable within some approximation. For example,  based on the theoretical framework within which the problem is modelled, one might know the \emph{explicit} integral kernel $a(x,y)$ of a Hilbert-Schmidt operator $A:L^2(\Omega,\ud\mu)\to L^2(\Omega,\ud\mu)$ on some $L^2$-space $\cH=L^2(\Omega,\ud\mu)$, and hence the explicit action $(Ah)(x)=\int_\Omega a(x,y)h(y)\ud \mu(y)$, and yet it might not be possible to write explicitly (exactly) the singular value decomposition for $A$.

Now, although well-defined inverse problems are in a sense trivial theoretically, as the existence and uniqueness of the solution is not of concern, to determine solutions by means of a (numerical) \emph{algorithm}  poses non-trivial questions on the computation of approximate solutions, their convergence to the `exact' one, the speed of convergence, etc.

In this respect, the possible infinite-dimensionality of the problem brings additional complication. We actually refer here to the case when \eqref{eq:gen} is \emph{genuinely} infinite-dimensional, meaning that not only $\dim\cH=\infty$, but also (see, e.g., \cite[Sect.~1.4]{schmu_unbdd_sa}) that $A$ is \emph{not} reduced to $A=A_1\oplus A_2$ by an orthogonal direct sum decomposition $\cH=\cH_1\oplus\cH_2$ with $\dim\cH_1<\infty$, $\dim\cH_2=\infty$, and $A_2=\mathbb{O}$ (for otherwise the effective problem would deal with a finite matrix).

As well known, in such infinite-dimensional setting, scientific computing demands a suitable discretisation of the problem \eqref{eq:gen} to \emph{truncated} finite-dimensional Hilbert spaces, in order to run the numerics with actual matrices.

\emph{Typically}, this is done by means of what is customarily referred to as a `\emph{projection method}', the essence of which is the following (details in Section \ref{sec:finitedimtrunc}). Two convenient, a priori known orthonormal systems $(u_n)_{n\in\mathbb{N}}$ and $(v_n)_{n\in\mathbb{N}}$ of $\cH$ are taken, and for each (large enough) integer $N$ the corresponding finite-dimensional `solution space' $\mathrm{span}\{u_1,\dots,u_N\}$ and `trial space' $\mathrm{span}\{v_1,\dots,v_N\}$ are constructed. Denoting by $P_N$ and $Q_N$ the orthogonal projection maps onto, respectively, the solution and the trial space, one replaces the original problem \eqref{eq:gen} with its $N$-dimensional truncation
\begin{equation}
 Q_N A P_N\, \widehat{f^{(N)}}\;=\;Q_N g
\end{equation}
in the unknown $\widehat{f^{(N)}}$ in the solution space. In other words, one searches for $\widehat{f^{(N)}}\in \mathrm{span}\{u_1,\dots,u_N\}$, solving 
\begin{equation}\label{eq:PGformulation}
Q_N(A \widehat{f^{(N)}} - g) \;=\;0\,.
\end{equation}
For large enough $N$, the resulting $\widehat{f^{(N)}}$'s are expected to produce reasonable approximations for the `exact' solution(s) $f$ to $Af=g$.

Some earlier works on truncation schemes include \cite{stummel1970,stummel1971,kranoselski-1972-approxsoll, vainikko1974,vainikko1978,vainikko1981,karma1992}. Their typical approach consists of studying the discrete convergence using a system of connection operators, this way encompassing a vast class of truncation schemes. Our approach is more specific in that we remain within the class of `\emph{projection methods}'. 

Surely the most popular class of such truncation schemes are the celebrated \emph{Petrov-Galerkin projection methods}  \cite[Chapter 4]{kranoselski-1972-approxsoll}, \cite[Chapter 5]{Saad-2003_IterativeMethods}, \cite[Chapter 9]{Atkinson-Han-TheoNumAnal2009} (and one simply speaks of a \emph{Galerkin method} when $u_n=v_n$ for all $n$). Such methods are customarily implemented under special assumptions that at this abstract level of generality can be presented as follows: 
\begin{itemize}
 \item both $(u_n)_{n\in\mathbb{N}}$ and $(v_n)_{n\in\mathbb{N}}$ are orthonormal \emph{bases} of $\cH$;
 \item moreover, $A$ and such bases are such that both the infinite-dimensional problem \eqref{eq:gen} and each truncated problem \eqref{eq:PGformulation} admit a unique solution, respectively $f$ and $\widehat{f^{(N)}}$, satisfying $\|\widehat{f^{(N)}}-f\|_{\cH}\to 0$ as $N\to\infty$.
\end{itemize}

All this is very familiar for certain classes of boundary value problems on $L^2(\Omega)$, where $\Omega$ is some domain in $\mathbb{R}^d$, the typical playground for Galerkin and Petrov-Galerkin finite element methods \cite{Ern-Guermond_book_FiniteElements,Quarteroni-book_NumModelsDiffProb}. In these cases $A$ is an \emph{unbounded} operator, say, of elliptic type \cite[Chapter 3]{Ern-Guermond_book_FiniteElements}, \cite[Chapter 4]{Quarteroni-book_NumModelsDiffProb}, of Friedrichs type \cite[Sect.~5.2]{Ern-Guermond_book_FiniteElements}, \cite{EGC,ABCE,Antonic-Erceg-Mich-2017}, of parabolic type \cite[Chapter 6]{Ern-Guermond_book_FiniteElements}, \cite[Chapter 5]{Quarteroni-book_NumModelsDiffProb}, of `mixed' (i.e., inducing saddle-point problems) type \cite[Sect.~2.4 and Chapter 4]{Ern-Guermond_book_FiniteElements}, etc. Such $A$'s are assumed to satisfy (and so do they in applications) some kind of coercivity, or more generally one among the various classical conditions that ensure the corresponding problem \eqref{eq:gen} to be well-posed, such as the Banach-Ne\u{c}as-Babu\v{s}ka Theorem or the Lax-Milgram Lemma \cite[Chapter 2]{Ern-Guermond_book_FiniteElements}. This makes the finite-dimensional truncation and the infinite-dimensional error analysis widely studied and well understood. In that context, in order for the finite-dimensional solutions to converge strongly, one requires stringent yet often plausible conditions \cite[Sect.~2.2-2.4]{Ern-Guermond_book_FiniteElements}, \cite[Sect.~4.2]{Quarteroni-book_NumModelsDiffProb}
\begin{itemize}
 \item[(a)] both on the truncation spaces, that need to approximate suitably well the ambient space $\cH$ (`\emph{approximability}', e.g., the interpolation capability of finite elements),
 \item[(b)] and on the behaviour of the reduced problems, that need admit solutions that are uniformly controlled by the data 
(`\emph{uniform stability}'),
 \item[(c)] and that are suitably good approximate solutions of the original problem (`\emph{asymptotic consistency}'), 
 \item[(d)] together with some suitable boundedness of the problem in appropriate topologies (`\emph{uniform continuity}').
\end{itemize}

As plausible as the above conditions are, they are \emph{not} matched by several other types of inverse problems of applied interest.

Mathematically this is the case, for instance, whenever $A$ does not have a `good' inverse, say, when $A$ is a compact operator on $\cH$ with arbitrarily small singular values.
Another meaningful example occurs when the the space for the finite-dimensional truncations is the so called Krylov subspace $\mathcal{K}(A,g)$ associated to the problem \eqref{eq:gen}, namely is spanned by the vectors $g,Ag,A^2g,\dots$ (a very popular choice indeed, as Krylov subspace methods are often counted among the `Top 10 Algorithms' of the 20th century \cite{Dongarra-Sullivan-Best10-2000,Cipra-SIAM-News}), and yet $\mathcal{K}(A,g)$ has a non-trivial orthogonal complement in $\cH$ and the solution $f$ to $Af=g$ fails to belong to the closure of $\mathcal{K}(A,g)$ -- we recently studied the phenomenon of `Krylov solvability', or lack thereof, in our recent works \cite{CMN-2018_Krylov-solvability-bdd,CM-2019_ubddKrylov}.

For these reasons, we rather adopt in this work the perspective of \emph{general projection methods}, thus moving \emph{outside} the standard working assumptions of the Petrov-Galerkin framework, in the sense that, while still projecting the infinite-dimensional problem to finite-dimensional truncations, yet
\begin{itemize}
 \item[(i)] the problem \eqref{eq:gen} is only assumed to be solvable;
 \item[(ii)] no a priori conditions are assumed that guarantee the truncated problems \eqref{eq:PGformulation} to be well-defined (they may thus have a multiplicity of solutions), let alone solvable (they may have no solution at all);
 \item[(iii)] not all the standard conditions (a)-(d) above are assumed (unlike for Petrov-Galerkin schemes), which ensured a convergence theory in the Hilbert norm of the error and/or the residual along the sequence of approximate solutions.
\end{itemize}

In particular, with reference to (iii), we carry on the point of view that error and residual may be controlled in a still informative way in some \emph{weaker} sense than the expected norm topology of the Hilbert space.

To this aim, we identify practically plausible sufficient conditions for the error or the residual to be small in such generalised senses and we discuss the \emph{mechanisms} why the same indicators may actually fail to vanish in norm. We also argue the genericity of such situations. We specialise our observations first on the case when $A$ is compact (Section \ref{sec:compactlinearinverse}) and then when $A$ is generically bounded (Section \ref{sec:boundedlinearinverse}). In the latter case, specifically when $A$ is self-adjoint and non-negative, we find instructive to re-interpret our findings in view of well-known convergence results for the conjugate gradient scheme, which is in fact a Krylov subspace projection method (Section \ref{sec:conjgrad}).


Last, in the concluding part of the work (Section \ref{sec:numerics}), we present a selection of numerical tests that illustrate the main features discussed theoretically.

In conclusion, we should like to stress that our main goal here is to present generic convergence mechanisms that occur, unavoidably (in the precise sense that we shall discuss), when one weakens some of the stringent working conditions that are typical of the Petrov-Galerkin schemes, thus giving rise to more general behaviours at the finite-dimensional truncated level and in the possible convergence to the ``exact'' solution(s).

In doing so, as mentioned already, we do not aim at a comprehensive theory of general projection methods for infinite-dimensional inverse problems, and we rather keep the point of view of presenting \emph{generic} features, phenomena, and difficulties that look `unavoidable' at the considered level of generality, and we discuss them through an amount of model examples that challenge the common intuition. In our intentions this should provide the setting for a future thorough analysis of classes of infinite-dimensional inverse problems.

\medskip

\textbf{General notation.} Besides further notation that will be declared in due time, we shall keep the following convention. $\cH$ denotes a real or complex Hilbert space, that will be separable throughout this note, with norm $\|\cdot\|_{\cH}$ and scalar product $\langle\cdot,\cdot\rangle$, which in the complex case is anti-linear in the first entry and linear in the second. Bounded operators on $\cH$ are tacitly understood to be linear and everywhere defined. $\|\cdot\|_{\mathrm{op}}$ denotes the corresponding operator norm. The space of bounded linear operators on $\cH$ is denoted with $\mathcal{B}(\cH)$. The spectrum of an operator $A$ is denoted by $\sigma(A)$. $\mathbbm{1}$ and $\mathbbm{O}$ are, respectively, the identity and the zero operator, meant as finite matrices or infinite-dimensional operators depending on the context. An upper bar denotes the complex conjugate $\overline{z}$ when $z\in\mathbb{C}$, and the norm closure $\overline{\mathcal{V}}$ of the span of the vectors in $\mathcal{V}$ when $\mathcal{V}$ is a subset of $\cH$. For $\psi,\varphi\in\cH$, by $|\psi\rangle\langle\psi|$ and $|\psi\rangle\langle\varphi|$ we shall denote the $\cH\to\cH$ rank-one maps acting respectively as $f\mapsto \langle \psi, f\rangle\,\psi$ and $f\mapsto \langle \varphi, f\rangle\,\psi$ on generic $f\in\cH$. For identities such as $\psi(x)=\varphi(x)$ in $L^2$-spaces we will tacitly understand the `for almost every $x$' specification in the equality.


\section{Finite-dimensional truncation}\label{sec:finitedimtrunc}

\subsection{Set up and notation}\label{sec:finitedimtrunc-setup}~

Let us start with revisiting, in the present abstract terms, the setting and the formalism for a generic projection scheme -- in the framework of Galerkin and Petrov-Galerkin methods this is customarily referred to as the `\emph{approximation setting}' \cite[Sect.~2.2.1]{Ern-Guermond_book_FiniteElements}.

Let $(u_n)_{n\in\mathbb{N}}$ and $(v_n)_{n\in\mathbb{N}}$ be two orthonormal \emph{systems} of the considered Hilbert space $\cH$. They need not be orthonormal \emph{bases}, although their completeness is often crucial for the goodness of the approximation.


The choice of $(u_n)_{n\in\mathbb{N}}$ and $(v_n)_{n\in\mathbb{N}}$ depends on the specific approach. In the framework of finite element methods they can be taken to be the global shape functions of the interpolation scheme \cite[Chapter 1]{Ern-Guermond_book_FiniteElements}. For Krylov subspace methods they are just the spanning vectors of the associated Krylov subspace \cite[Chapter 2]{Liesen-Strakos-2003}.

Correspondingly, for each $N\in\mathbb{N}$, the orthonormal projections in $\cH$ respectively onto $\mathrm{span}\{u_1,\dots,u_N\}$ and $\mathrm{span}\{v_1,\dots,v_N\}$ shall be
\begin{equation}\label{eq:defPNQN}
 P_N\;:=\;\sum_{n=1}^N|u_n\rangle\langle u_n|\,,\qquad Q_N\;:=\;\sum_{n=1}^N|v_n\rangle\langle v_n|\,.
\end{equation}

Associated to a given inverse problem $Af=g$ in $\cH$ as \eqref{eq:gen}, one considers the finite-dimensional truncations induced by $P_N$ and $Q_N$, that is, for each $N$, the problem to find solutions $\widehat{f^{(N)}}\in P_N\cH$ to the equation
\begin{equation}\label{eq:Ntrunc_inf}
 (Q_N A P_N) \widehat{f^{(N)}} \;=\; Q_N g\,.
\end{equation}
In \eqref{eq:Ntrunc_inf} $Q_N g=\sum_{n=1}^N \langle v_n,g\rangle v_n$ is the datum and $\widehat{f^{(N)}}=\sum_{n=1}^N \langle u_n,\widehat{f^{(N)}}\rangle u_n$ is the unknown, and the compression $Q_N A P_N$ is only non-trivial as a map from $P_N\cH$ to $Q_N\cH$, its kernel containing at least the subspace $(\mathbbm{1}-P_N)\cH$.


There is an obvious and irrelevant degeneracy (which is infinite when $\dim\cH=\infty$) in \eqref{eq:Ntrunc_inf} when it is regarded as a problem on the whole $\cH$. The actual interest towards \eqref{eq:Ntrunc_inf} is the problem resulting from the identification $P_N\cH\cong\mathbb{C}^N\cong Q_N\cH$, in terms of which $P_N f\in\cH$ and $Q_N g\in\cH$ are canonically identified with the vectors
\begin{equation}\label{eq:fNgN}
 f_N\;=\;
 \begin{pmatrix}
  \langle u_1,f\rangle \\
  \vdots \\
  \langle u_N,f\rangle
 \end{pmatrix}\in\mathbb{C}^N\,,\qquad
 g_N\;=\;
 \begin{pmatrix}
  \langle v_1,g\rangle \\
  \vdots \\
  \langle v_N,g\rangle
 \end{pmatrix}\in\mathbb{C}^N\,,
\end{equation}
and $Q_N A P_N$ with a $\mathbb{C}^N\to\mathbb{C}^N$ linear map represented by the $N\times N$ matrix $A_N=(A_{N;ij})_{i,j\in\{1,\dots,N\}}$
\begin{equation}\label{eq:matrixAN}
 A_{N;ij}\;=\;\langle v_i,Q_N AP_N u_j\rangle\,.
\end{equation}
The matrix $A_N$ is what in the framework of finite element methods for partial differential equations is customarily referred to as the `\emph{stiffness matrix}'.

We shall call the linear inverse problem
\begin{equation} \label{eq:Ntrunc}
 A_N f^{(N)}\;=\;g_N
\end{equation}
with datum $g_N\in\mathbb{C}^N$ and unknown $f^{(N)}\in\mathbb{C}^N$, and matrix $A_N$ defined by \eqref{eq:matrixAN}, the \emph{$N$-dimensional truncation} of the original problem $Af=g$.

Let us stress the meaning of the present notation.
\begin{itemize}
 \item $Q_N A P_N$, $P_Nf$, and $Q_N g$ are objects (one operator and two vectors) referred to the whole Hilbert space $\cH$, whereas $A_N$, $f^{(N)}$, $f_N$, and $g_N$ are the analogues referred now to the space $\mathbb{C}^N$.
 \item Moreover, the subscript in $A_N$, $f_N$, and $g_N$ indicates that the components of such objects are precisely the corresponding components, up to order $N$, respectively of $A$, $f$, and $g$, with respect to the tacitly declared bases $(u_n)_{n\in\mathbb{N}}$ and $(v_n)_{n\in\mathbb{N}}$, through formulas \eqref{eq:fNgN}-\eqref{eq:matrixAN}. 
 \item As opposite, the superscript in $f^{(N)}$ indicates that the components of the $\mathbb{C}^N$-vector $f^{(N)}$ are not necessarily to be understood as the first $N$ components of the $\cH$-vector $f$ with respect to the basis $(u_n)_{n\in\mathbb{N}}$, and in particular for $N_1<N_2$ the components of $f^{(N_1)}$ are not a priori equal to the first $N_1$ components of $f^{(N_2)}$. In fact, if $f\in\cH$ is a solution to $Af=g$, it is evident from obvious counterexamples that in general the truncations $A_N$, $f_N$, $g_N$ do \emph{not} satisfy the identity $A_N f_N=g_N$, whence the notation $f^{(N)}$ for the unknown in \eqref{eq:Ntrunc}.
 \item Last, for a $\mathbb{C}^N$-vector $f^{(N)}$ the notation $\widehat{f^{(N)}}$ indicates a vector in $\cH$ whose first $N$ components, with respect to the basis $(u_n)_{n\in\mathbb{N}}$, are precisely those of $f^{(N)}$, all others being zero. Thus, as pedantic as it looks, $f^{(N)}=(\widehat{f^{(N)}})_N$ and  $f_N=(\widehat{f_N})_N$, and of  course in general $f\neq \widehat{f_N}$.
\end{itemize}

With $A$, $g$, $(u_n)_{n\in\mathbb{N}}$, and $(v_n)_{n\in\mathbb{N}}$ explicitly known, the truncated problem \eqref{eq:Ntrunc} is explicitly formulated and, being finite-dimensional, it is suited for numerical algorithms.

This poses the general question on \emph{whether the truncated problem itself is solvable, and whether its exact or approximate solution $f^{(N)}$ is close to the exact solution $f$ and in which (possibly quantitative) sense}. 

Let us elaborate more on these two issues in the following two subsections.

\subsection{Singularity of the truncated problem}~

It is clear, first of all, that the question of the singularity of the truncated problem \eqref{eq:Ntrunc} makes sense here \emph{eventually in $N$}, meaning for all $N$'s that are large enough. For a \emph{fixed} value of $N$ the truncation might drastically alter the problem so as to make it manifestly non-informative as compared to $Af=g$, such alteration then disappearing for larger values.

%
%

Yet, even when the solvability of $A_N f^{(N)}=g_N$ 
is inquired eventually in $N$, it is no surprise that the answer is generically negative.

\begin{example}\label{example:RN}
 That the matrix $A_N$ may remain singular for arbitrary $N$ even when the operator $A$ is injective can be seen, for example, with the truncation of the weighted right-shift operator $\mathcal{R}=\sum_{n\in\mathbb{Z}}\sigma_n|e_{n+1}\rangle\langle e_n|$ on $\ell^2(\mathbb{\mathbb{Z}})$ onto the $\mathrm{span}\{e_{-N},e_{-N+1},\dots,e_{N-1},e_N\}$, where $(e_n)_{n\in\mathbb{Z}}$ is the canonical orthonormal basis of $\ell^2(\mathbb{\mathbb{Z}})$ and $\sigma\equiv(\sigma_n)_{n\in\mathbbm{N}}$ is a given bounded sequence with $0<\sigma_{n+1}<\sigma_n$ $\forall n\in\mathbb{N}$ and $\lim_{n\to\infty}\sigma_n=0$.
 $\mathcal{R}$ is indeed compact and injective. However, 
\begin{equation}\label{eq:RN}
 \mathcal{R}_{2N+1}\;:=\;
 \begin{pmatrix}
  0        & \cdots   & \cdots & \cdots   & 0\\
  \sigma_{-N} & 0        & \cdots & \cdots   & 0 \\
  0        & \sigma_{-N+1} & 0      & \cdots   & 0\\
  \vdots   & \vdots   & \ddots & \ddots   & 0\\
  0        & 0        & \cdots & \sigma_{N-1} & 0
 \end{pmatrix}
\end{equation}
is singular irrespectively of $N\in\mathbb{N}$, with $\ker\mathcal{R}_{2N+1}=\mathrm{span}\{e_N\}$. (See Lemma \ref{lem:bad_truncations} below for a more general perspective on such an example.) 
Noticeably, $\sigma(\mathcal{R}_{2N+1})=\{0\}$ for each $N$, (explicitly, the vector with all null components but the last one is precisely the eigenvector spanning the kernel of $\mathcal{R}_{2N+1}$):
the erroneous eigenvalue zero that arises in the limit $N\to\infty$ is a typical manifestation of the phenomenon known as \emph{spectral pollution}.
\end{example}

\begin{remark}
 In the same spirit, it is not difficult to produce variations of the above example where the matrix $A_N$ is, say, alternatingly singular and non-singular as $N\to\infty$. 
 Of course, on the other hand, it may also well happen that the truncated matrix is always non-singular: the truncation of the multiplication operator $M=\sum_{n=1}^\infty\frac{1}{n}\,|e_n\rangle\langle e_n|$ on $\ell^2(\mathbb{N})$ with respect to  $(e_n)_{n\in\mathbb{N}}$ yields the matrix $M_N=\mathrm{diag}(1,\frac{1}{2},\dots,\frac{1}{N})$, which is a $\mathbb{C}^N\to\mathbb{C}^N$ bijection for every $N$.
\end{remark}

In fact, `bad' truncations  are always possible, as the following mechanism shows.

\begin{lemma}\label{lem:bad_truncations}
Let $\cH$ be a separable Hilbert space with $\dim\cH=\infty$, and let $A\in\mathcal{B}(\cH)$. There always exist two orthonormal systems $(u_n)_{n \in \mathbb{N}}$ and $(v_n)_{n \in \mathbb{N}}$ of $\cH$ such that the corresponding truncated matrix $A_N$ defined as in \eqref{eq:matrixAN} is singular for every $N\in\mathbb{N}$.
\end{lemma}

\begin{proof}
Let us pick an arbitrary orthonormal system $(u_n)_{n \in \mathbb{N}}$ and construct the other system $(v_n)_{n \in \mathbb{N}}$ inductively. When $N=1$, it suffices to choose $v_1$ such that $v_1\perp A u_1$ and $\|v_1\|_{\cH}=1$. Let now $(v_n)_{n\in\{1,\dots,N-1\}}$ be an orthonormal system in $\cH$ satisfying the thesis up to the order $N-1$ and let us construct $v_N$ so that $(v_n)_{n\in\{1,\dots,N\}}$ satisfies the thesis up to order $N$. To this aim, let us show that a choice of $v_N$ is always possible so that the final row in the matrix $A_N$ has all zero entries. In fact, $(A_N)_{ij} = (Q_NAP_N)_{ij}=\langle v_i, Au_j \rangle$ for $i\in\{1,\dots ,N - 1\}$ and $j\in\{1,\dots,N\}$ and in order for $\langle v_N, Au_j \rangle = 0$ for $j \in\{1,\cdots, N\}$ it suffices to take
\[
 v_N\perp\mathrm{ran}(AP_N)\,,\qquad  v_N\perp\mathrm{ran}\, Q_{N-1}\,,\qquad \|v_N\|_{\cH}=1\,,
\]
where $P_N$ and $Q_{N-1}$ are the orthogonal projections defined in \ref{eq:defPNQN}. Since $\mathrm{ran}(AP_N)$ and $\mathrm{ran} \,Q_{N-1}$ are finite-dimensional subspaces of $\cH$, there is surely a vector $v_N\in \cH$ with the above properties.
\end{proof}

The occurrence described by Lemma \ref{lem:bad_truncations} may happen both with an orthogonal and with an oblique projection scheme, namely both when $P_N=Q_N$ and when $P_N\neq Q_N$ eventually in $N$.

In the standard framework of (Petrov-)Galerkin methods such an occurrence is \emph{explicitly prevented} by suitable assumptions on $A$, typically coercivity \cite[Sect.~2.2]{Ern-Guermond_book_FiniteElements}, \cite[Sect.~4.1]{Quarteroni-book_NumModelsDiffProb}, and by similar assumptions at the truncated level (e.g., \cite[Sect.~3.2, Sect.~4.2]{Ern-Guermond_book_FiniteElements}), so as to ensure the truncated problem \eqref{eq:Ntrunc} to be solvable for all $N$.

As in our discussion we do not exclude a priori such an occurrence, we are compelled to regard $f^{(N)}$ as an approximate solution to the truncated problem, in the sense that
\begin{equation}\label{eq:Af-g-e}
 A_N f^{(N)}\;=\;g_N+\varepsilon^{(N)}\qquad\textrm{for some $\varepsilon^{(N)}\in\mathbb{C}^N$}.
\end{equation}
(We write $\varepsilon^{(N)}$ and not $\varepsilon_N$ because there is no reason to claim that the residual $\varepsilon^{(N)}$ in the $N$-dimensional problem is the actual truncation for every $N$ of the same infinite-dimensional vector $\varepsilon\in\cH$.)

It would be desirable that $\varepsilon^{(N)}$ is indeed small and asymptotically vanishing with $N$, or even that $\varepsilon^{(N)}=0$ for $N$ large enough, as is the case in some applications. Morally (up to passing to the weak formulation of the inverse problem), this is the assumption of \emph{asymptotic consistency} naturally made for approximations by Galerkin methods \cite[Definition 2.15 and Theorem 2.24]{Ern-Guermond_book_FiniteElements}.

In the present abstract context it is worth remarking that an assumption of that sort it is motivated by the following property, whose proof is postponed to Sect.~\ref{sec:boundedlinearinverse}.

\begin{lemma}\label{lem:ANfN-gNsmall}
 Let $A\in\mathcal{B}(\cH)$ and $g\in\mathrm{ran}A$. Let $A_N$ and $g_N$ be defined as in \eqref{eq:fNgN}-\eqref{eq:matrixAN} above, for some orthonormal bases $(u_n)_{n \in \N}$, $(v_n)_{n \in \N}$ of $\cH$. Then there always exists a sequence $(f^{(N)})_{N\in\mathbb{N}}$ such that
 \[
  f^{(N)}\in\mathbb{C}^N\qquad\textrm{ and }\qquad \lim_{N\to \infty}\|A_N f^{(N)}-g_N\|_{\mathbb{C}^N}=0\,.
 \]
\end{lemma}

In other words, there do exist approximate solutions $f^{(N)}$ to \eqref{eq:Ntrunc} actually satisfying \eqref{eq:Af-g-e} with $\|\varepsilon^{(N)}\|_{\mathbb{C}^N}\to 0$ as $N\to\infty$, so that \eqref{eq:Af-g-e} is asymptotically consistent.

%



\subsection{Convergence of the truncated problems}~

Depending on the context, the vanishing of the displacement $f-\widehat{f^{(N)}}$ as $N\to\infty$ is monitored in various forms with respect to the Hilbert norm of $\cH$, the two most typical ones are the norms of the \emph{infinite-dim\-en\-sion\-al error} $\mathscr{E}_N$ and the \emph{infinite-dim\-en\-sion\-al residual} $\mathfrak{R}_N$, defined respectively as
\begin{equation}\label{eq:error-residual}
 \begin{split}
   \mathscr{E}_N\;&:=\;f-\widehat{f^{(N)}} \\
   \mathfrak{R}_N\;&:=\;A(f-\widehat{f^{(N)}})\;=\;g-A\,\widehat{f^{(N)}}\,.
 \end{split}
\end{equation}
We qualify them as `infinite-dimensional', although we shall drop this additional nomenclature when no confusion arises, in order to distinguish them from the error and residual at fixed $N$, which may be indexed by the number of steps in an iterative algorithm.

A first evident obstruction to the actual vanishing of $\mathscr{E}_N$ when $\dim\cH=\infty$ is the use of a \emph{non-complete} orthonormal system $(u_n)_{n\in\mathbb{N}}$, that is, such that $\mathrm{span}\{u_n\,|\,n\in\mathbb{N}\}$ is not dense in $\cH$.


Truncations with respect to a potentially non-complete orthonormal system might appear unwise, but in certain contexts are natural. One is the vast framework of the Krylov subspace methods \cite{Liesen-Strakos-2003}, where one searches for approximate solutions among the linear combinations of the vectors $g,Ag,A^2g,\dots$ and hence performs the truncation with respect to an orthonormal basis of the \emph{Krylov subspace} 
\begin{equation}\label{eq:defKrylov}
  \mathcal{K}(A,g) \;:=\;  \mathrm{span} \{ A^kg\,|\,k\in\mathbb{N}_0\}
\end{equation}
associated to $A\in\mathcal{B}(\cH)$ and $g\in\cH$. Obviously, when $\dim\mathcal{K}(A,g)=\infty$ the subspace $\mathcal{K}(A,g)$ is neither open nor closed in $\cH$. Its closure can be the whole $\cH$, but also just a \emph{proper} closed subspace of $\cH$.

\begin{example}\label{example:Krylov_dense_or_not}~

 \begin{itemize}
  \item[(i)] For the weighted right-shift operator $\mathcal{R}=\sum_{n=1}^\infty\sigma_{n}|e_{n+1}\rangle\langle e_n|$ on $\ell^2(\mathbb{N})$ and the vector $g=e_{2}$, $\overline{\mathcal{K}(\mathcal{R},e_{2})}=\mathrm{span}\{e_1\}^\perp$, a proper subspace of $\ell^2(\mathbb{N})$. Therefore, the exact solution $f=\frac{1}{\sigma_1}e_1$ to the inverse problem $\mathcal{R}f=g=e_2$ is orthogonal to $\overline{\mathcal{K}(R,e_{m+1})}$, and the truncated problem can only produce approximate solutions $ \widehat{f^{(N)}}\in\mathrm{span}\{e_2,e_3,\dots\}$, whence $\widehat{f^{(N)}}\perp f$ and $\|\widehat{f^{(N)}}-f\|_{\cH}\geqslant\frac{1}{\sigma_1}$.
  \item[(ii)] For the (compact, injective) Volterra integral operator $h\mapsto Vh$ on $L^2[0,1]$ with $(Vh)(x):=\int_0^x h(y)\ud y$,  and the function $g=\mathbf{1}$ (the constant function with value 1), the functions $Vg,V^2g,V^3g,\dots$ are (multiples of) the polynomials $x,x^2,x^3,\dots$, therefore $\mathcal{K}(V,g)$ is the space of polynomials on $[0,1]$, which is dense in $L^2[0,1]$.
 \end{itemize}
\end{example}

In standard (Petrov-)Galerkin methods an occurrence as in Example \ref{example:Krylov_dense_or_not}(i) is ruled out by an ad hoc `\emph{approximability'} assumption \cite[Definition 2.14 and Theorem 2.24]{Ern-Guermond_book_FiniteElements} that can be rephrased as the request that $(u_n)_{n\in\mathbb{N}}$ is indeed an orthonormal \emph{basis} of $\cH$.

The approximability property is known to fail in situations of engineering interest, as is the case for the failure of the Lagrange finite elements in differential problems for electromagnetism \cite[Sect.~2.3.3]{Ern-Guermond_book_FiniteElements}. 

Even when (complete) orthonormal bases of $\cH$ are employed for the truncation, another consequence of the infinite dimensionality is the possibility that error and residual are asymptotically small only in some weaker sense than the customary norm topology of $\cH$.

It is therefore natural to contemplate, next to the strong ($\cH$-norm) convergence vanishing $\|\mathfrak{R}_N\|_{\cH}\to 0$, resp., $\|\mathscr{E}_N\|_{\cH}\to 0$, where obviously
\begin{equation}\label{ENtozeroRNtozero}
 \|\mathfrak{R}_N\|_{\cH}\;\leqslant\;\|A\|_{\mathrm{op}}\,\|\mathscr{E}_N\|_{\cH}\,,
\end{equation}
also the \emph{weak} convergence $\mathfrak{R}_N\rightharpoonup 0$ or $\mathscr{E}_N\rightharpoonup 0$, or even the \emph{component-wise convergence}, namely the vanishing of each component of the vector $\mathfrak{R}_N$ or $\mathscr{E}_N$ with respect to the considered basis.

As we shall elaborate further on:

\emph{a strong control such as $\|\mathfrak{R}_N\|_{\cH}\to 0$ or $\|\mathscr{E}_N\|_{\cH}\to 0$ is \emph{not} generic in truncation schemes that do not obey the more stringent assumptions of Petrov-Galerkin projection methods, and only holds under specific a priori conditions on the linear inverse problem.}

On the other hand, even a mere component-wise vanishing of $\mathscr{E}_N$ is in many respects already satisfactorily informative, for in this case each component of  $\widehat{f^{(N)}}$ (with respect to the basis $(u_n)_{n\in\mathbb{N}}$) approximates the corresponding component of the exact solution $f$.

Thus, as already recalled in the Introduction, for elliptic boundary value problems the standard Galerkin finite element method produces a \emph{strong} vanishing of the error, provided that two crucial conditions are satisfied, namely a \emph{careful} choice of the truncation space and the \emph{coercivity} of the differential operator \cite[Sect.~4.2.3]{Quarteroni-book_NumModelsDiffProb}: when this is the case, the vanishing rate depends on the truncation basis and the regularity of the solution. More generally \cite[Sect.~2.3.1]{Ern-Guermond_book_FiniteElements}, standard Petrov-Galerkin methods give rise to a strong convergence of the approximate solution under the simultaneous validity of uniform stability, uniform boundedness and asymptotic consistency of the linear problem, and approximability by means of the chosen truncation spaces.
When the differential operator is non-coercive, additional sufficient conditions have been studied for the stability of the truncated problem and for the quasi optimality of the discretization scheme \cite{Caorsi-Fernandes-Raffetto-2000,Buffa-Christiansen-2003,Buffa-2005}.

On a related scenario, special classes of linear ill-conditioned problems (rank-deficient and discrete ill-posed problems) can be treated with regularisation methods in which the solution is stabilised \cite{Tikhonov-Arsenin_Ill-posed_problems,Hansen-Illposed-1998}. The most notable regularisation methods, namely the Tikhonov-Phillips method, the Landweber-Fridman iteration method, and the truncated singular value decomposition, produce indeed a strongly vanishing error \cite{Groetsch-1984,Louis-1989}. Yet, when the linear inverse problem $Af=g$ is governed by an infinite-rank compact operator $A$, it can be seen that the conjugate gradient method, as well as $\alpha$-processes (in particular, the method of steepest descent), although having regularising properties, may have strongly divergent error and residual in the presence of noise \cite{Eicke-Louis-Plato-1990} and one is forced to consider weaker forms of convergence. In fact, in \cite{Eicke-Louis-Plato-1990} the presence of component-wise convergence is also alluded to.

\section{The compact linear inverse problem}\label{sec:compactlinearinverse}

Let us now examine, within the framework elaborated in the previous Section, the abstract truncation and convergence scheme for \emph{compact} inverse problems.

When the operator $A$ is compact on $\cH$, it is standard to refer to its `\emph{singular value decomposition}' \cite[Theorem VI.17]{rs1} as the expansion
\begin{equation}\label{eq:singValDec}
  A\;=\;\sum_{n\in J} \sigma_n\,|\psi_n\rangle\langle\varphi_n|\,,
 \end{equation}
 where $n$ runs over $J:=\{1,\dots,M\}$, with $M<\infty$ or $M=\infty$,  $(\varphi_n)_{n}$ and $(\psi_n)_{n}$ are two orthonormal systems of $\cH$, $\sigma_n\geqslant\sigma_{n+1}>0$ for all $n$, and $\sigma_n\to 0$, so that the above series, if infinite, converges in operator norm. 
 
 Let us recall that the injectivity of $A$ is tantamount as $(\varphi_n)_{n\in\mathbb{N}}$ being an orthonormal basis and that $\overline{\mathrm{ran} A}=\cH$ if an only if $(\psi_n)_{n\in\mathbb{N}}$ is an orthonormal basis.

 Let us consider the inverse problem \eqref{eq:gen} for compact \emph{and injective} $A$, $g\in\mathrm{ran} A$, and $\dim\cH=\infty$. The problem is well-defined: there exists a unique $f\in\cH$ such that $Af=g$.

The compactness of $A$ has two noticeable consequences here. First, since $\dim\cH=\infty$, $A$ is invertible on its range only, and cannot have an everywhere defined bounded inverse: $\mathrm{ran}A$ can be dense in $\cH$, as in the case of the Volterra operator on $L^2[0,1]$, or dense in a closed proper subspace of $\cH$, as for the weighted right-shift on $\ell^2(\mathbb{N})$.

Furthermore, $A$ and its compression (in the usual meaning of Sect.~\ref{sec:finitedimtrunc-setup}) are close in a robust sense, as the following standard Lemma shows.

\begin{lemma}\label{lem:Op-norm-conv}
With respect to an infinite-dimensional separable Hilbert space $\cH$, let $A:\cH\to\cH$ be a compact operator and let $(u_n)_{n\in\mathbb{N}}$ and $(v_n)_{n\in\mathbb{N}}$ be two orthonormal bases of $\cH$. Then
\begin{equation}
 \|A-Q_NAP_N\|_{\mathrm{op}}\;\xrightarrow[]{\;N\to\infty\;}\;0\,,
\end{equation}
$P_N$ and $Q_N$ being as usual the orthogonal projections \eqref{eq:defPNQN}.
\end{lemma}

\begin{proof}
 Upon splitting
 \[
  A-Q_NAP_N\;=\;(A-Q_N A) + Q_N(A-AP_N)
 \]
 it suffices to prove that $\|A-AP_N\|_{\mathrm{op}}\xrightarrow[]{N\to\infty}0$ and $\|A-Q_N A\|_{\mathrm{op}}\xrightarrow[]{N\to\infty}0$.
 Let us prove the first limit (the second being completely analogous).
 
 Clearly, it is enough to prove that $\|A-AP_N\|_{\mathrm{op}}$ vanishes assuming further that $A$ has finite rank. Indeed, the difference $(A-AP_N)-(\widetilde{A}-\widetilde{A}P_N)$, where $\widetilde{A}$ is a finite-rank approximant of the compact operator $A$, is controlled in operator norm by $2\|A-\widetilde{A}\|_{\mathrm{op}}$ and hence can be made arbitrarily small.
 
 Thus, we consider non-restrictively $A=\sum_{k=1}^M\sigma_k|\psi_k\rangle\langle \varphi_k|$ for some integer $M$, where $(\varphi_k)_{k=1}^M$ and $(\psi_k)_{k=1}^M$ are two orthonormal systems, and $0<\sigma_M\leqslant\cdots\leqslant\sigma_1$. Now, for a generic $\xi=\sum_{n=1}^\infty\xi_n v_n\in\cH$ one has
 \[
  \begin{split}
   \big\|(A&-AP_N)\xi\big\|_{\cH}^2\;=\;\Big\|\sum_{k=1}^M \sigma_k\Big( \sum_{n=N+1}^\infty \xi_n \langle\varphi_k,v_n\rangle\Big)\psi_k\Big\|_{\cH}^2 \\
   &=\;\sum_{k=1}^M \sigma_k^2\,\Big| \sum_{n=N+1}^\infty \xi_n \langle\varphi_k,v_n\rangle\,\Big|^2\;\leqslant\;\|\xi\|_{\cH}^2\sum_{k=1}^M \sigma_k^2\,\big\|(\mathbbm{1}-P_N)\varphi_k\big\|_{\cH}^2\,,
  \end{split}
 \]
 therefore 
 \[
  \big\|A-AP_N\big\|_{\mathrm{op}}^2\;\leqslant\;M\,\sigma_1^2\cdot\max_{k\in\{1,\dots,M\}}\big\|(\mathbbm{1}-P_N)\varphi_k\big\|_{\cH}^2 \;\xrightarrow[]{\;N\to\infty\;}\;0\,,
 \]
since the above maximum is taken over $M$ (hence, finitely many) quantities, each of which vanishes as $N\to\infty$.
\end{proof}

In the following Theorem we describe the generic behaviour of well-defined compact inverse problem.

\begin{theorem}\label{prop:compact1}
 Consider 
 \begin{itemize}
  \item the linear inverse problem $Af=g$ in a separable Hilbert space $\cH$ for some compact and injective $A:\cH\to\cH$  and some $g\in\mathrm{ran}A$;
  \item the finite-dimensional truncation $A_N$ obtained by compression with respect to the orthonormal bases $(u_n)_{n\in\mathbb{N}}$ and $(v_n)_{n\in\mathbb{N}}$ of $\cH$.
 \end{itemize}
Let $(f^{(N)})_{N\in\mathbb{N}}$ be a sequence of approximate solutions to the truncated problems in the quantitative sense
\[
 A_N f^{(N)}\;=\;g_N+\varepsilon^{(N)}\,,\qquad f^{(N)},\varepsilon^{(N)}\in\mathbb{C}^N\,,\qquad \|\varepsilon^{(N)}\|_{\mathbb{C}^N}\;\xrightarrow[]{\;N\to\infty\;}\;0
\]
for every (sufficiently large) $N$. If $\widehat{f^{(N)}}$ is $\cH$-norm bounded uniformly in $N$, then
\[
   \|\mathfrak{R}_N\|_{\cH}\to 0\qquad\textrm{and}\qquad \mathscr{E}_N\rightharpoonup 0\qquad\textrm{as}\;\; N\to \infty\,.
 \]
%
%
\end{theorem}

\begin{proof}
   We split
 \[\tag{*}
  \begin{split}
   A\widehat{f^{(N)}}- g\;&=\;(A-Q_NAP_N)\widehat{f^{(N)}} \\
   &\qquad +\; Q_NAP_N\widehat{f^{(N)}} - Q_Ng \\
   &\qquad +\; Q_Ng-g\,.
  \end{split}
 \]
By assumption, $\|Q_Ng-g\|_{\cH}\xrightarrow[]{\;N\to\infty\;}0$ and
 \[
  \begin{split}
   \|Q_NAP_N\widehat{f^{(N)}} - Q_Ng\|_{\cH}\;&=\; \| A_N f^{(N)}-g_N\|_{\mathbb{C}^N} \\
   &=\;\|\varepsilon^{(N)}\|_{\mathbb{C}^N}\;\xrightarrow[]{\;N\to\infty\;}\;0\,.
  \end{split}
 \]
 Moreover, Lemma \ref{lem:Op-norm-conv} and the uniform boundedness of $\widehat{f^{(N)}}$ imply
 \[
  \|(A-Q_NAP_N)\widehat{f^{(N)}}\|_{\cH}\;\leqslant\;\|A-Q_NAP_N\|_{\mathrm{op}}\,\|\widehat{f^{(N)}}\|_{\cH}\;\xrightarrow[]{\;N\to\infty\;}\;0
 \]
 Plugging the three limits above into (*) proves $\|\mathfrak{R}_N\|_{\cH}\to 0$.

 Next, in terms of the singular value decomposition \eqref{eq:singValDec} of $A$, where now $(\varphi_n)_{n\in\mathbb{N}}$ is an orthonormal basis of $\cH$, $(\psi_n)_{n\in\mathbb{N}}$ is an orthonormal system, and  $0<\sigma_{n+1}\leqslant\sigma_n$ for all $n\in J$, we write
 \[
  \widehat{f^{(N)}}\;=\;\sum_{n\in\mathbb{N}} f^{(N)}_n\varphi_n\,,\qquad \widehat{f}\;=\;\sum_{n\in\mathbb{N}} f_n\varphi_n\,,
 \]
 whence
 \[
  0\;=\; \lim_{N\to\infty}\|A \widehat{f^{(N)}}-g\|_{\cH}^2\;=\,\lim_{N\to\infty}\:\sum_{n\in J}\sigma_n^2\big| f^{(N)}_n-f_n\big|^2\,.
 \]
 Then necessarily $\widehat{f^{(N)}}$ converges to $f$ component-wise:
 \[
  \langle u_n,\widehat{f^{(N)}}\rangle\xrightarrow[]{\;N\to\infty\;}\langle u_n,f\rangle \quad\forall n\in\mathbb{N}\,.
 \]
  
  On the other hand, $\widehat{f^{(N)}}$ is uniformly bounded in $\cH$, $\sup_{N\in\mathbb{N}}\|\eta_N\|_{\cH}<+\infty$. As well known, latter two conditions are equivalent to the fact that $\widehat{f^{(N)}}$ converges to $f$ weakly  ($\mathscr{E}_N\rightharpoonup 0$).
%
%
 %
%
%
%
%
\end{proof}

Theorem \ref{prop:compact1} provides sufficient conditions for some form of vanishing of the error and the residual.
The key assumptions are: 
\begin{itemize}
 \item \emph{injectivity} of $A$, 
 \item \emph{asymptotic solvability of the truncated problems}, i.e., asymptotic smallness of the finite-dimensional residual $ A_N f^{(N)}-g_N$,
 \item \emph{uniform boundedness of the approximate solutions} $f^{(N)}$.
\end{itemize}
In fact, injectivity was only used in the analysis of the error in order to conclude $\mathscr{E}_N\rightharpoonup 0$; instead, the conclusion $\|\mathfrak{R}_N\|_{\cH}\to 0$ follows irrespectively of injectivity.

It is also worth observing that both the assumption of asymptotic solvability of the truncated problems (which is inspired to Lemma \ref{lem:ANfN-gNsmall}, as commented already) and the assumption of uniform boundedness of the $f^{(N)}$'s are finite-dimensional in nature: their check, from the practical point of view of scientific computing, requires a control of explicit $N$-dimensional vectors for a `numerically satisfactory' amount of integers $N$.

To further understand the impact of such assumptions, a few remarks are in order.

\begin{remark}[Genericity]\label{rem:strongRes_weakErr}
 Under the conditions of Theorem \ref{prop:compact1}, the occurrence of the \emph{strong} vanishing of the residual ($\|\mathfrak{R}_N\|_{\cH}\to 0$) and the \emph{weak} vanishing of the error ($\mathscr{E}_N\rightharpoonup 0$) as $N\to\infty$ is indeed a \emph{generic behaviour}. For example, the compact inverse problem $\mathcal{R}f=0$ in $\ell^2(\mathbb{N})$ associated with the weighted right-shift $\mathcal{R}$  has exact solution $f=0$. The truncated problem $\mathcal{R}_N f^{(N)}=0$ with respect to $\mathrm{span}\{e_1,\dots,e_N \}$, where
 \begin{equation}\label{eq:RRN}
 \mathcal{R}_N\;=\;
 \begin{pmatrix}
  0        & \cdots   & \cdots & \cdots   & 0\\
  \sigma_1 & 0        & \cdots & \cdots   & 0 \\
  0        & \sigma_2 & 0      & \cdots   & 0\\
  \vdots   & \vdots   & \ddots & \ddots   & 0\\
  0        & 0        & \cdots & \sigma_{N-1} & 0
 \end{pmatrix},
\end{equation}
 is solved by the $\mathbb{C}^N$-vectors whose first $N-1$ components are zero, i.e., $\widehat{f^{(N)}}=e_N$. The sequence $(\widehat{f^{(N)}})_{N\in\mathbb{N}}\equiv(e_N)_{N\in\mathbb{N}}$ converges weakly to zero in $\ell^2(\mathbb{N})$, whence indeed $\mathscr{E}_N\rightharpoonup 0$, and also, by compactness, $\|\mathfrak{R}_N\|_{\cH}\to 0$. However, $\|\mathscr{E}_N\|_{\cH}=1$ for every $N$, thus the error cannot vanish in the $\cH$-norm.
\end{remark}

\begin{remark}[`Bad' approximate solutions]\label{rem:Res_small_Err_large}
 The example considered in Remark \ref{rem:strongRes_weakErr} is also instructive to understand that generically one may happen to select `bad' approximate solutions $\widehat{f^{(N)}}$ such that, despite the `good' strong convergence
 $\|A_N f^{(N)}-g_N\|_{\mathbb{C}^N}\to 0$, have the unsatisfactory feature $\|f^{(N)}\|_{\mathbb{C}^N}=\|\widehat{f^{(N)}}\|_{\cH}\to +\infty$: this is the case if one chooses, for instance, $\widehat{f^{(N)}}=N e_N$. Thus, the uniform boundedness of $\widehat{f^{(N)}}$ in $\cH$ required in Theorem \ref{prop:compact1} is \emph{not} redundant. (This is also consistent with the fact that whereas by compactness $\widehat{f^{(N)}}\rightharpoonup f$ implies $\| A\widehat{f^{(N)}}- Af\|\to 0$, yet the opposite implication is not true in general.)
\end{remark}

\begin{remark}[The density of $\mathrm{ran} A$ does not help]
 The genericity of convergence discussed in Remarks \ref{rem:strongRes_weakErr} and \ref{rem:Res_small_Err_large}
 is not improved by additionally requiring that $\overline{\mathrm{ran} A}=\cH$.
 For instance, the inverse problem considered in Example \ref{example:RN} involves an operator that is compact, injective, and with dense range, but the finite-dimensional projections produce \emph{for every $N$} the matrix \eqref{eq:RN} that is singular and for which, therefore, all the considerations of Remarks \ref{rem:strongRes_weakErr} and \ref{rem:Res_small_Err_large} can be repeated verbatim.
\end{remark}


\begin{remark}[`Bad' truncations and `good' truncations]\label{rem:goodtruncation}
 We saw in Lemma \ref{lem:bad_truncations} that `bad' truncations (i.e., leading to matrices $A_N$ that are, eventually in $N$, all singular) are always possible. On the other hand, there always exists a ``good'' choice for the truncation -- although such a choice might not be identifiable explicitly -- which makes the infinite-dimensional residual and error vanish in norm and with no extra assumption of uniform boundedness on the approximate solutions. For instance, in terms of the singular value decomposition \eqref{eq:singValDec} of $A$, it is enough to choose
 \[
  (u_n)_{n\in\mathbb{N}}\;=\;(\varphi_n)_{n\in\mathbb{N}}\,,\qquad  (v_n)_{n\in\mathbb{N}}\;=\;(\psi_n)_{n\in\mathbb{N}}\,,
\]
 in which case $Q_NAP_N=\sum_{n=1}^N\sigma_n|\psi_n\rangle\langle\varphi_n|$ and $A_N=\mathrm{diag}(\sigma_1,\dots,\sigma_N)$, and for given $g=\sum_{n\in\mathbb{N}}g_n\psi_n$ one has  $\widehat{f^{(N)}}=\sum_{n=1}^N\frac{g_n}{\sigma_n}\,\varphi_n$, where the sequence $(\frac{g_n}{\sigma_n})_{n\in\mathbb{N}}$ belongs to $\ell^2(\mathbb{N})$ owing to the assumption $g\in\mathrm{ran}A$,  
 whence
 \[
 \|f-\widehat{f^{(N)}}\|_{\cH}^2\;=\;\sum_{n=N+1}^\infty\Big|\frac{g_n}{\sigma_n} \Big|^2\;\xrightarrow[]{\;N\to\infty\;}\;0\,.
 \]
\end{remark}

\section{The bounded linear inverse problem}\label{sec:boundedlinearinverse}

It is instructive to compare the findings of the previous Section with the more general case of a \emph{bounded} (not necessarily compact) linear inverse problem.

When $\dim\cH=\infty$ and a generic bounded linear operator $A:\cH\to\cH$ is compressed (in the usual sense of Sect.~\ref{sec:finitedimtrunc}) between the spans of the first $N$ vectors of the orthonormal bases $(u_n)_{n\in\mathbb{N}}$ and $(v_n)_{n\in\mathbb{N}}$, then surely $Q_N A P_N\to A$  as $N\to\infty$ in the strong operator topology, that is, $\|Q_N A P_N\psi- A\psi\|_{\cH}\xrightarrow[]{\;N\to\infty\;}0$ $\forall\psi\in\cH$, yet the convergence \emph{may fail to occur in the operator norm}.

The first statement is an obvious consequence of the inequality
\[
 \|(A-Q_N A P_N)\psi\|_{\cH}\;\leqslant\;\|(\mathbbm{1}-Q_N)A\psi\|_{\cH}+\|A\|_{\mathrm{op}}\,\|\psi-P_N\psi\|_{\cH}
\]
valid for any $\psi\in\cH$. The lack of operator norm convergence is also clear for any non-compact $A$: indeed, the operator norm limit of finite-rank operators can only be compact.

For this reason, the control of the infinite-dimensional inverse problem in terms of its finite-dimensional truncated versions is in general less strong.

As a counterpart of Theorem \ref{prop:compact1} above, let us discuss the following generic behaviour of \emph{well-posed} bounded linear inverse problems.

\begin{theorem}\label{thm:boundedinvprob}
 Consider 
 \begin{itemize}
  \item the linear inverse problem $Af=g$ in a Hilbert space $\cH$ for some bounded and injective $A:\cH\to\cH$ and some $g\in\cH$;
  \item the finite-dimensional truncation $A_N$ obtained by compression with respect to the orthonormal bases $(u_n)_{n\in\mathbb{N}}$ and $(v_n)_{n\in\mathbb{N}}$ of $\cH$.
 \end{itemize}
Let $(f^{(N)})_{N\in\mathbb{N}}$ be a sequence of approximate solutions to the truncated problems in the quantitative sense
\[
 A_N f^{(N)}\;=\;g_N+\varepsilon^{(N)}\,,\qquad f^{(N)},\varepsilon^{(N)}\in\mathbb{C}^N\,,\qquad \|\varepsilon^{(N)}\|_{\mathbb{C}^N}\xrightarrow[]{\;N\to\infty\;}0
\]
for every (sufficiently large) $N$. 
\begin{itemize}
 \item[(i)] The residuals $\mathfrak{R}_N$ vanish strongly, or weakly or component-wise as $N\to\infty$ if and only if so do the vectors $(A-Q_NAP_N)\widehat{f^{(N)}}$.
 \item[(ii)] In particular, if $\widehat{f^{(N)}}$ converges strongly in $\cH$, equivalently, if $\|f^{(N)}-f^{(M)}\|_{C^{\max\{N,M\}}}\xrightarrow[]{\;N,M\to\infty\;}0$, then
  \[
  \|\mathscr{E}_N\|_{\cH}\to 0\qquad\textrm{and}\qquad \|\mathfrak{R}_N\|_{\cH}\to 0\qquad\textrm{as }N\to\infty\,.
 \]
\end{itemize}
\end{theorem}

\begin{proof}
 Since
 \[\tag{*}
  \begin{split}
   A\widehat{f^{(N)}}- g\;&=\;(A-Q_NAP_N)\widehat{f^{(N)}} \\
   &\qquad +\; Q_NAP_N\widehat{f^{(N)}} - Q_Ng \\
   &\qquad +\; Q_Ng-g\,,
  \end{split}
 \]
 and since by assumption $\|Q_Ng-g\|_{\cH}\xrightarrow[]{\;N\to\infty\;}0$ and 
 \[
  \begin{split}
   \|Q_NAP_N\widehat{f^{(N)}} - Q_Ng\|_{\cH}\;&=\; \| A_N f^{(N)}-g_N\|_{\mathbb{C}^N} \\
   &=\;\|\varepsilon^{(N)}\|_{\mathbb{C}^N}\;\xrightarrow[]{\;N\to\infty\;}\;0\,,
  \end{split}
 \]
 then the strong, weak, or component-wise vanishing of $A\widehat{f^{(N)}}- g$ is tantamount as, respectively, the strong, weak, or component-wise vanishing of $(A-Q_NAP_N)\widehat{f^{(N)}}$.

 Since in addition $\|\widehat{f^{(N)}}-\widetilde{f}\|_{\cH}\xrightarrow[]{\;N\to\infty\;} 0$ for some $\widetilde{f}\in\cH$, then
 \[
  \begin{split}
  \|(A-Q_NAP_N)\widehat{f^{(N)}}\|_{\cH}\;&\leqslant\;\|(A-Q_NAP_N)\widetilde{f}\|_{\cH}+2\,\|A\|_{\mathrm{op}}\,\|\widetilde{f}-\widehat{f^{(N)}}\|_{\cH} \\
  &\qquad\xrightarrow[]{\;N\to\infty\;}\;0
  \end{split} 
  \]
 (the first summand in the r.h.s.~above vanishing due to the operator strong convergence $Q_NAP_N\to A$), and  (*) thus implies $ \|\mathfrak{R}_N\|_{\cH}=\|A\widehat{f^{(N)}}-g\|_{\cH}\xrightarrow[]{\;N\to\infty\;}0$.

 Thus, part (i) is proved. Under the assumption of part (ii), one has $A\widehat{f^{(N)}}\to g$ (owing to (i)) and $A\widehat{f^{(N)}}\to A\widetilde{f}$ (by continuity), whence $A\widetilde{f}=g=Af$ and also (by injectivity) $f=\widetilde{f}$. This shows that $\|\mathscr{E}_N\|_{\cH}=\|f-\widehat{f^{(N)}}\|_{\cH}=\|\widetilde{f}-\widehat{f^{(N)}}\|_{\cH}\to 0$. 
\end{proof}

We observe that also here injectivity was only used in the analysis of the error, whereas it is not needed to conclude that $\|\mathfrak{R}_N\|_{\cH}\to 0$.

As compared to Theorem \ref{prop:compact1}, Theorem \ref{thm:boundedinvprob} now relies on the following hypotheses:
\begin{itemize}
 \item \emph{injectivity} of $A$, 
 \item \emph{asymptotic solvability of the truncated problems},
 \item \emph{convergence of the approximate solutions} $f^{(N)}$.
\end{itemize}
Once again, the last two assumptions are finite-dimensional in nature and from a numerical perspective their check involves a control of explicit $N$-dimensional vectors for satisfactorily many $N$'s.

\begin{remark}
In passing from a (well-defined) \emph{compact} to a generic (well-defined) \emph{bounded} inverse problem, one has to strengthen the hypothesis of uniform boundedness of the $\widehat{f^{(N)}}$'s to their actual strong convergence, in order for the residual $\mathfrak{R}_N$ to vanish strongly (in which case, as a by-product, also the error $\mathscr{E}_N$ vanishes strongly). In the compact case, $A-Q_NAP_N\to\mathbb{O}$ in operator norm (Lemma \ref{lem:Op-norm-conv}), and it suffices that the $\widehat{f^{(N)}}$'s be uniformly bounded (or, in principle, have increasing norm $\|\widehat{f^{(N)}}\|_{\cH}$ compensated by the vanishing of $\|A-Q_NAP_N\|_{\mathrm{op}}$), in order for $\|\mathfrak{R}_N\|_{\cH}\to 0$. 
\end{remark}

When instead the sequence of the $\widehat{f^{(N)}}$'s \emph{does not} converge strongly, in general one has to expect only weak vanishing of the residual, $\mathfrak{R}_N\rightharpoonup 0$, which in turn prevents the error to vanish strongly -- for otherwise $\|\mathscr{E}_N\|_{\cH}\to 0$ would imply $\|\mathfrak{R}_N\|_{\cH}\to 0$, owing to \eqref{ENtozeroRNtozero}. The following example shows such a possibility.

\begin{example}\label{example:convergence_needed}
 For the right-shift $R=\sum_{n=1}^\infty|e_{n+1}\rangle\langle e_n|$ on $\ell^2(\mathbb{N})$, the inverse problem $Rf=g$ with $g=0$ admits the unique solution $f=0$. The truncated finite-dimensional problems induced by the bases $(u_n)_{n\in\mathbb{N}}=(v_n)_{n\in\mathbb{N}}=(e_n)_{n\in\mathbb{N}}$
 is governed by the sub-diagonal matrix $R_N$ given by \eqref{eq:RRN} with $\sigma_j=1$ $\forall j\in\mathbb{N}$.
 Let us consider the sequence $(\widehat{f^{(N)}})_{N\in\mathbb{N}}$ with $\widehat{f^{(N)}}:=e_N$ for each $N$. 
 Then:
 \begin{itemize}
  \item $R_N f^{(N)}=0=g_N$ (the truncated problems are solved exactly),
  \item $\widehat{f^{(N)}}\rightharpoonup 0$ (only weakly, not strongly),
  \item $\mathfrak{R}_N=g-R\widehat{f^{(N)}}=-e_{N+1}\rightharpoonup 0$ (only weakly, not strongly).
 \end{itemize}
\end{example}

Of course, what discussed so far highlights features of generic bounded inverse problems (as compared to compact ones). Ad hoc analyses for special classes of bounded inverse problems are available and complement the picture of Theorem \ref{thm:boundedinvprob}. This is the case, to mention one, when $A$ is an \emph{algebraic operator}, namely $p(A)=\mathbbm{O}$ for some polynomial $p$ (which includes finite-rank $A$'s) and one treats the inverse problem with the generalised minimal residual method (GMRES) \cite{Gasparo-Papini-Pasquali-2008}.

In retrospect, based on the reasoning of this Section we prove Lemma \ref{lem:ANfN-gNsmall}.

\begin{proof}[Proof of Lemma \ref{lem:ANfN-gNsmall}]
 Let $f$ satisfy $Af=g$. The sequence $(f^{(N)})_{N\in\mathbb{N}}$ defined by
 \[
  f^{(N)}\;:=\;(P_N f)_N\;=\;f_N\qquad \textrm{(that is, $\;\widehat{f^{(N)}}=P_N f$)}
 \]
 does the job, as a straightforward consequence of the fact, argued already at the beginning of this Section, that $Q_N A P_N\to A$ strongly in the operator topology. Indeed, 
 \[
  \begin{split}
   \|A_N &f^{(N)}-g_N\|_{\mathbb{C}^N}\;=\;\|Q_NAP_N\widehat{f^{(N)}}-Q_N g\|_{\cH} \\
   &\leqslant\;\|(Q_NAP_N -A) f\|_{\cH}+\|(1-Q_N)Af\|_{\cH}\,.
  \end{split}
 \]
The strong limit yields the conclusion.
\end{proof}

\section{Comparison to conjugate gradient schemes}\label{sec:conjgrad}

In this Section we further discuss the scope of Theorems \ref{prop:compact1} (compact case) and \ref{thm:boundedinvprob} (bounded case) in application to \emph{conjugate gradient} schemes for bounded, self-adjoint, positive semi-definite linear inverse problems \cite[Chapt.~7]{Engl-Hanke-Neubauer-1996}, \cite[Sect.~9.3.2]{Ern-Guermond_book_FiniteElements}, \cite[Sect.~7.2.2]{Quarteroni-book_NumModelsDiffProb}. Thus, throughout this Section $A=A^*\in\mathcal{B}(\cH)$ and $A\geqslant\mathbb{O}$, i.e., $\langle h,Ah\rangle\geqslant 0$ $\forall h\in\cH$.

In particular, we recognise the key role played by the assumption of \emph{uniform boundedness} (or even \emph{strong convergence}) of the finite-dimensional approximants.

When $A$ is taken as above, and $g\in\mathrm{ran}A$, the problem $Af=g$ admits solution(s) in $\cH$, which form the closed, convex, non-empty `solution manifold'
\begin{equation}
 \mathcal{S}(A,g)\;:=\;\{f\in\cH\,|\,Af=g\}\,.
\end{equation}
Clearly, if $A$ is injective, which in this case amounts to $A$ being positive definite, then $ \mathcal{S}(A,g)$ only consists of one element. Moreover, any $f\in\mathcal{S}(A,g)$ is variationally characterised as
\begin{equation}\label{eq:energy-functional}
 \Phi[f]\;=\;\min_{h\in\cH}\Phi[h]\,,\qquad \Phi[h]\;:=\; \mathfrak{Re}\big(\langle h,Af\rangle-2\langle h,g\rangle\big)\,.
\end{equation}
As $A\geqslant\mathbb{O}$, the above minimisation can be equivalently taken with respect to the functional $\widetilde{\Phi}[h]:=\langle h-f,A(h-f)\rangle$, the so-called `energy (semi-)norm' of the difference $h-f$.

Now, in the framework of conjugate gradient schemes one builds a sequence $(f^{[N]})_{N\in\mathbb{N}_0}$, the so-called `\emph{conjugate gradient iterates}', by taking the initial guess $f^{[0]}$ arbitrarily in $\cH$, and $f^{[N]}$, for $N\geqslant 1$, to be the minimiser of the problem
 \begin{equation}\label{eq:conjgradmin}
  \min_{h\in\mathcal{Q}_N}\Phi[h]\,,\qquad 
  \begin{array}{rcl}
   \mathcal{Q}_N\!\! & := &\!\! \{f^{[0]}\}+\mathrm{span}\{r_0,A r_0,\dots, A^{N-1} r_0\} \\
   r_0\!\! & := & \!\!Af^{[0]}-g\,.
  \end{array}
 \end{equation}
 Here `iterates' refers to the fact that the $f^{[N]}$'s can be equivalently obtained by means of certain iterative algebraic procedures 
 \cite{Hestenes-Stiefer-ConjGrad-1952,Nemirovskiy-Polyak-1985,Saad-2003_IterativeMethods,Liesen-Strakos-2003}.
 
 The notation for the superscript in $f^{[N]}$ is chosen to avoid confusion with the special meaning already reserved to $f^{(N)}$ and $\widehat{f^{(N)}}$ in the general setting of Sect.~\ref{sec:finitedimtrunc-setup}, although it is clear that the $f^{[N]}$'s here are to be considered on the same conceptual footing as the $\widehat{f^{(N)}}$'s, that is, they can be naturally regarded as \emph{approximate} solutions, expected to satisfy $Af^{[N]}\approx g$ in a suitable sense. This is suggested by the very construction \eqref{eq:conjgradmin} and the variational characterisation \eqref{eq:energy-functional} of the solution(s) $f$. In fact, with the non-restrictive choice $f^{[0]}=0$ the minimisation manifold $\mathcal{Q}_N$ in \eqref{eq:conjgradmin} is precisely the $N$-th order Krylov subspace associated with $A$ and $g$.

 That the above expectation is correct is expressed in rigorous terms by Theorem \ref{thm:Nemirovski} below, a classical result by Nemirovskiy and Polyak \cite{Nemirovskiy-Polyak-1985} (with a precursor version by Kammerer and Nashed \cite{Kammerer-Nashed-1972}), and discussed in more recent terms in \cite[Sect.~7.2]{Engl-Hanke-Neubauer-1996}, \cite[Sect.~3.2]{Hanke-ConjGrad-1995}, and \cite{CM-Nemi-unbdd-2019}. In order to state it, let us introduce the map $\mathcal{P}_{\mathcal{S}}:\cH\to\mathcal{S}(A,g)$ that associates to a point $h\in\cH$ the nearest point $\mathcal{P}_{\mathcal{S}}h$ of the solution manifold. Then one has the following.
 
 
 \begin{theorem}\label{thm:Nemirovski}\emph{(Nemirovskiy and Polyak \cite[Theorem 7]{Nemirovskiy-Polyak-1985}.)}~
 
 \noindent Let $A=A^*\in\mathcal{B}(\cH)$ with $\langle h,Ah\rangle\geqslant 0$ $\forall h\in\cH$, and let the sequence $(f^{[N]})_{N\in\mathbb{N}_0}$ in $\cH$ be defined by \eqref{eq:conjgradmin} above. Then
  \begin{equation}\label{eq:Nemirovski}
\lim_{N\to\infty}\|f^{[N]}-\mathcal{P}_{\mathcal{S}}f^{[0]}\|_{\cH}\,=\,0\,,
  \end{equation}
and moreover, for every $\gamma>0$ such that the problem $A^{\gamma/2}u=f^{[0]}-\mathcal{P}_{\mathcal{S}}f^{[0]}$ admits a solution $u\in\cH$, one has
 \begin{equation}\label{eq:Nemirovski2}
  \|f^{[N]}-\mathcal{P}_{\mathcal{S}}f^{[0]}\|_{\cH}\;\leqslant\;C_{f^{[0]},\gamma} \cdot N^{-\gamma}
 \end{equation}
 for some constant $C_{f^{[0]},\gamma}>0$ depending on $f^{[0]}$ and $\gamma$.
 \end{theorem}

 When $A$ is injective and hence $\mathcal{S}(A,g)$ only consists of the unique solution $f$ to $Af=g$, \eqref{eq:Nemirovski} reads $\|f^{[N]}-f\|_{\cH}\to 0$ as $N\to\infty$. In the analogy with the analysis of Theorem \ref{thm:boundedinvprob}, the sequence of approximate solutions is convergent and the error $\mathscr{E}_N$ indeed vanishes strongly, and so does, necessarily, the residual $\mathfrak{R}_N$. We can thus understand Theorem \ref{thm:Nemirovski} in view of our Theorem \ref{thm:boundedinvprob}.

\section{Numerical tests: effects of changing the truncation basis}\label{sec:numerics}

In this final Section we examine some of the features discussed theoretically so far through a few numerical tests concerning different choices of the truncation bases. We employed a Legendre, complex Fourier, and a Krylov basis to truncate the problems.

The two model operators that we considered are the Volterra operator $V$ in $L^2[0, 1]$  and the self-adjoint multiplication operator $M:L^2[1,2]\to L^2[1,2]$, $\psi\mapsto x\psi$. We examined the following two inverse problems.

\medskip

\underline{\textbf{First problem:}} $Vf_1=g_1$, with $g_1(x)=\frac{1}{2}x^2$.

The problem has unique solution
 \begin{equation}\label{eq:solV}
  f_1(x)\;=\;x\,,\qquad \|f_1\|_{L^2[0,1]}\;=\;\frac{1}{\sqrt{3}}\;\simeq\;0.5774
 \end{equation}
 and $f_1$ \emph{is} a Krylov solution, i.e., $f_1\in\overline{\mathcal{K}(V,g)}$, although $f_1\notin\mathcal{K}(V,g)$. To prove the first fact, let us observe that $\mathcal{K}(V,g)$ is spanned by the monomials $x^2,x^3,x^4,\dots$, i.e., $\mathcal{K}(V,g) = \{x^2p\,|\, p\textrm{ is a polynomial on } [0,1]\}$; therefore, if $h\in\mathcal{K}(V,g)^\perp$, then $0=\int_0^1\overline{h(x)}\,x^2p(x)\,\ud x$ for any polynomial $p$; the $L^2$-density of polynomials on $[0,1]$ implies necessarily that $x^2h=0$, whence also $h=0$; this proves that $\mathcal{K}(V,g)^\perp=\{0\}$ and hence $\overline{\mathcal{K}(V,g)}=L^2[0,1]$. The fact that $f_1\notin\mathcal{K}(V,g)$ follows from $f_1(x)=x^2\cdot\frac{1}{x}$ and $\frac{1}{x}\notin L^2[0,1]$.
 

\medskip

\underline{\textbf{Second problem:}} $Mf_2=g_2$, with $g_2(x)=x^2$.

The problem has unique solution
\begin{equation}\label{eq:solM}
  f_2(x)\;=\;x\,,\qquad \|f_2\|_{L^2[1,2]}\;=\;\sqrt{\frac{7}{3}}\;\simeq\;1.5275
 \end{equation}
 and $f_2$ \emph{is} a Krylov solution. Indeed, $\mathcal{K}(M,g) = \{x^2p\,|\, p\textrm{ is a polynomial on } [1,2]\}$ and $\overline{\mathcal{K}(M,g)} = \{x^2h(x)\,|\, h \in L^2[1,2]\}=L^2[1,2]$, whence $f_2\in\overline{\mathcal{K}(M,g)}$ and $f_2\notin\mathcal{K}(M,g)$.

\medskip

\begin{figure}[!t]
  \centering
  \begin{subfigure}[b]{\textwidth}
    \includegraphics[width = 0.3\textwidth]{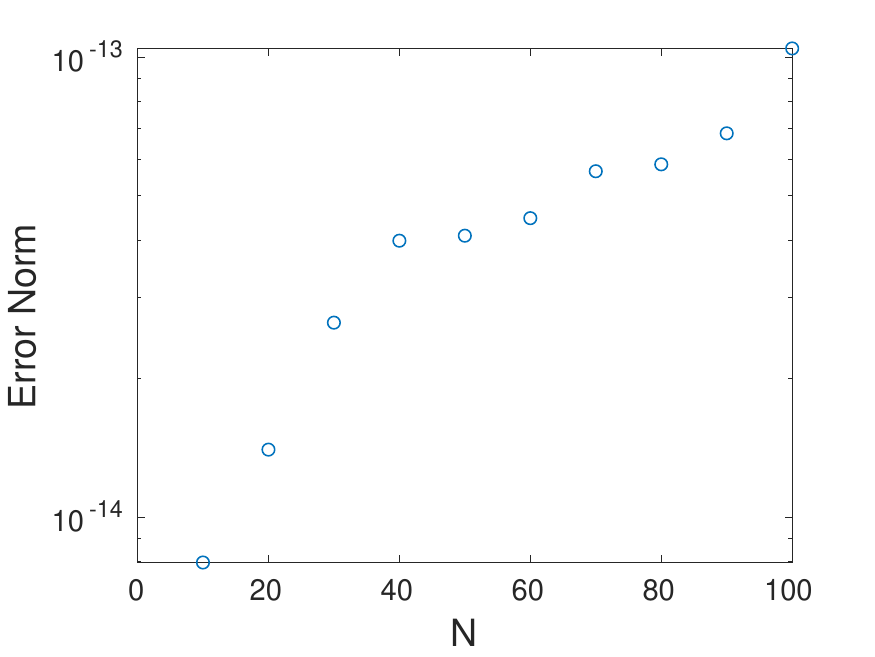} \quad
    \includegraphics[width = 0.3\textwidth]{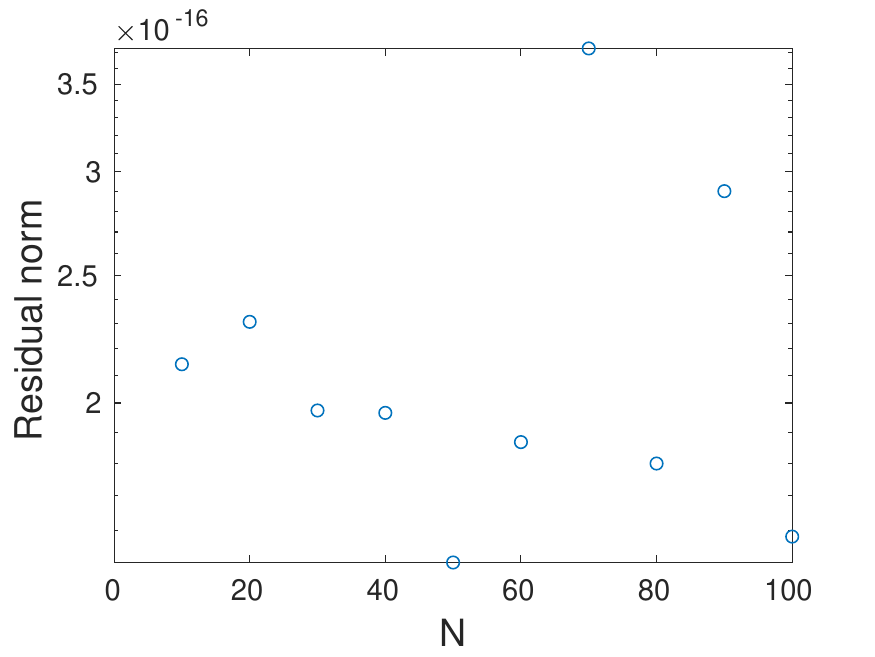} \quad
    \includegraphics[width = 0.3\textwidth]{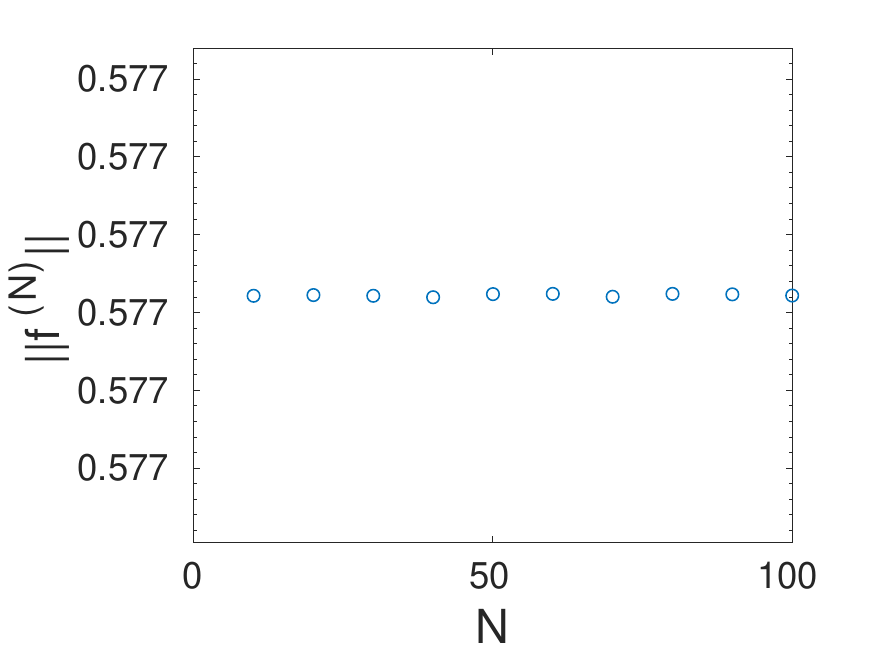}
    \caption{Legendre basis truncation}
  \end{subfigure}
  \begin{subfigure}[b]{\textwidth}
    \includegraphics[width = 0.3\textwidth]{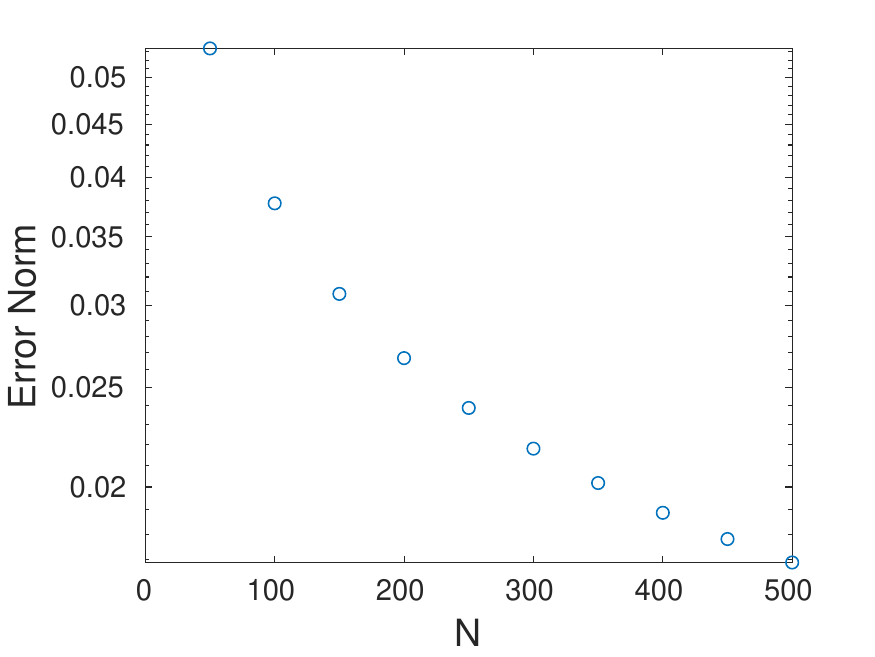} \quad
    \includegraphics[width = 0.3\textwidth]{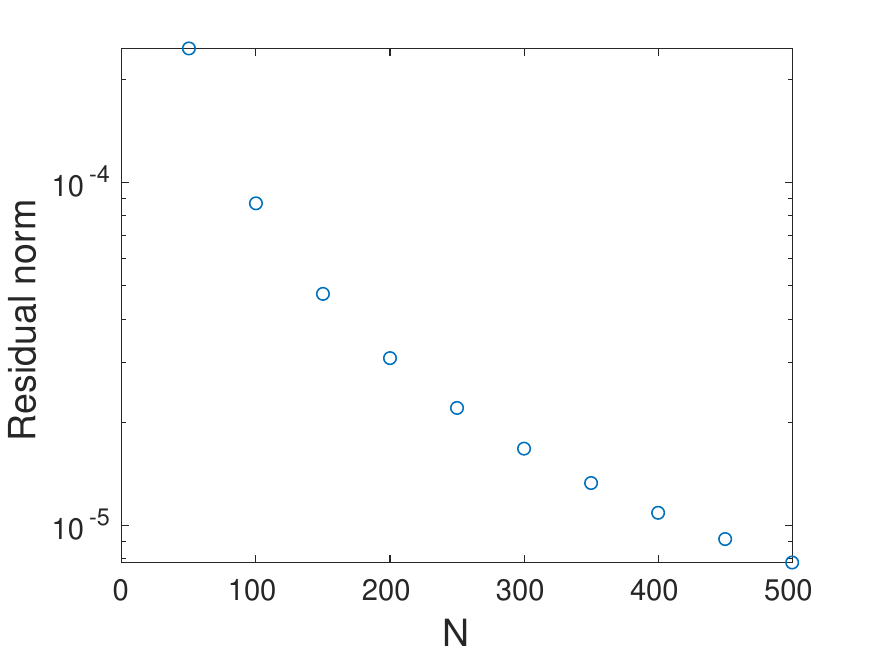} \quad
    \includegraphics[width = 0.3\textwidth]{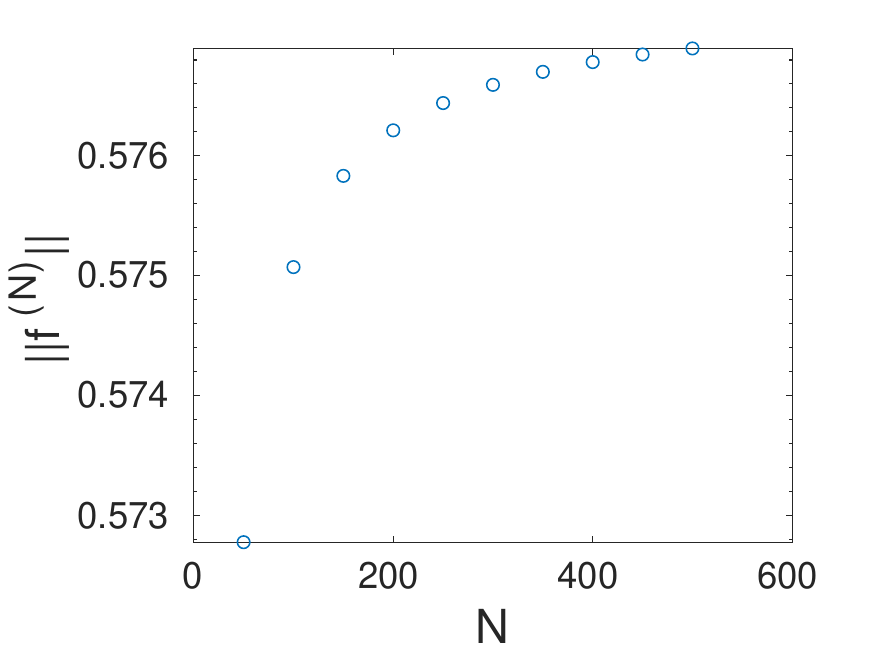}
    \caption{Complex Fourier basis truncation}
  \end{subfigure}
  \begin{subfigure}[b]{\textwidth}
    \includegraphics[width = 0.3\textwidth]{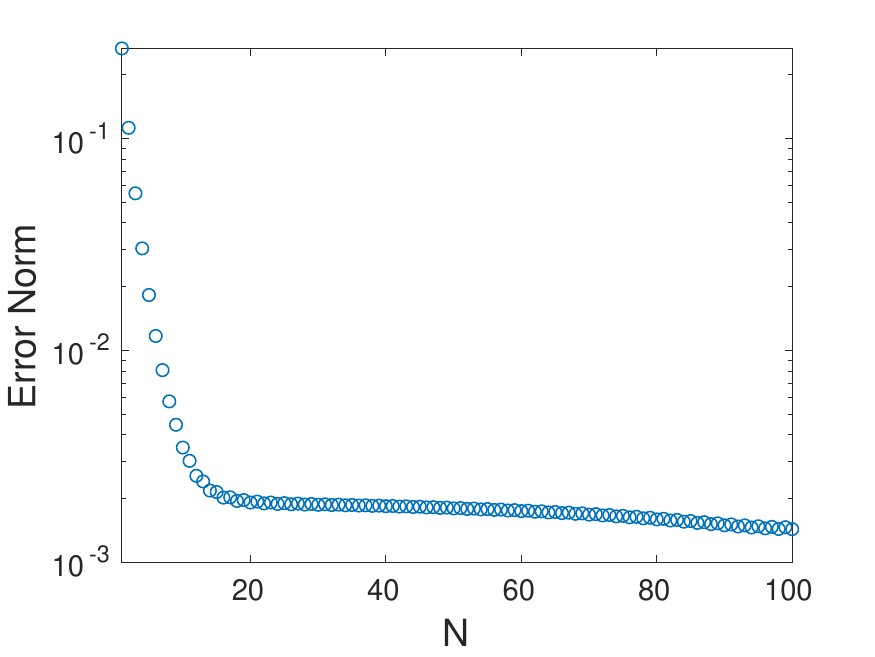} \quad
    \includegraphics[width = 0.3\textwidth]{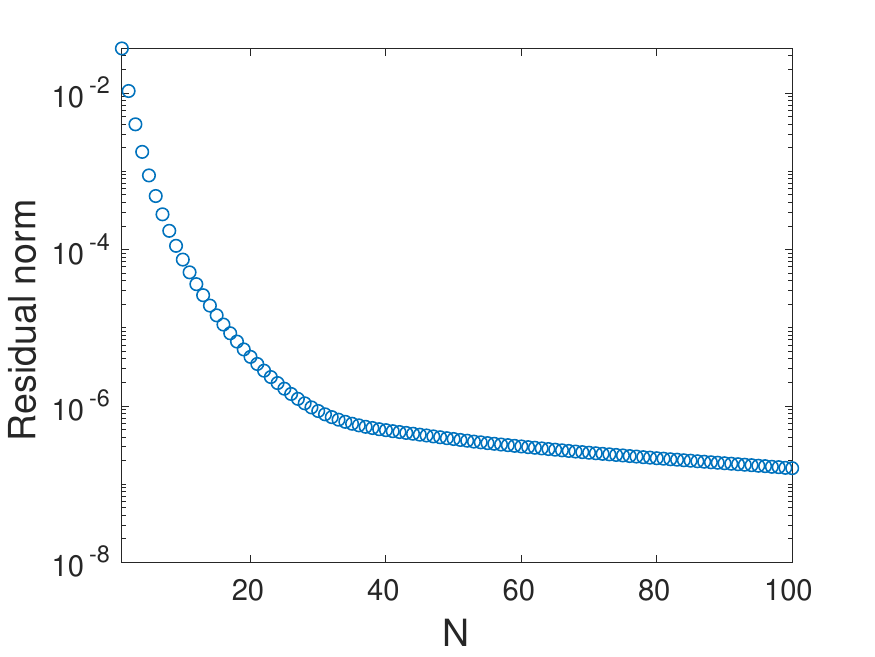} \quad
    \includegraphics[width = 0.3\textwidth]{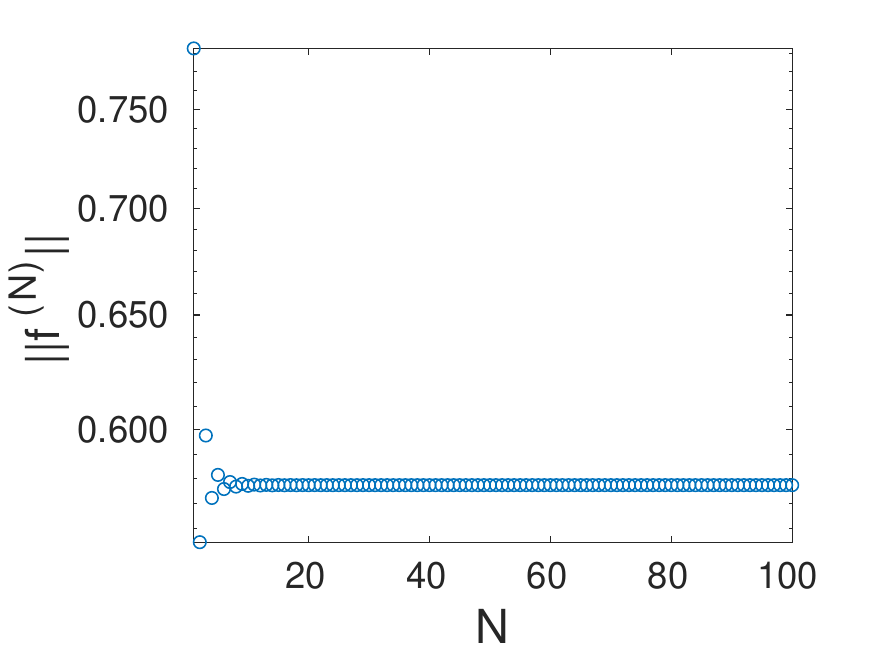} \quad
    \caption{Krylov basis truncation}
  \end{subfigure}
  \caption{Norm of the infinite-dimensional error and residual, and of the approximated solution for the Volterra inverse problem truncated with the Legendre, complex Fourier, and Krylov bases.} \label{fig:Volt_basis}
\end{figure}

\begin{figure}[!t]
  \centering
  \begin{subfigure}[b]{\textwidth}
    \includegraphics[width = 0.3\textwidth]{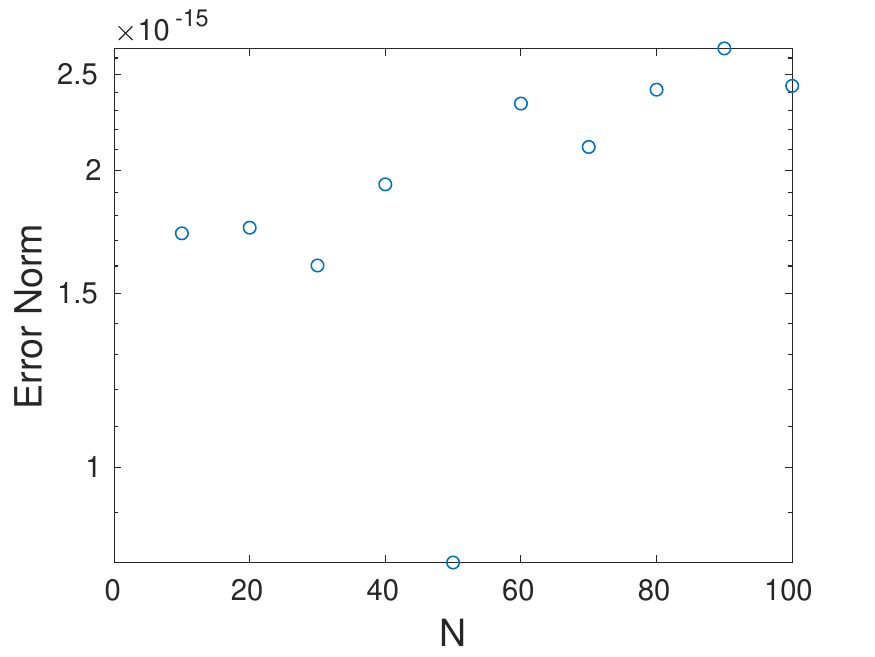} \quad
    \includegraphics[width = 0.3\textwidth]{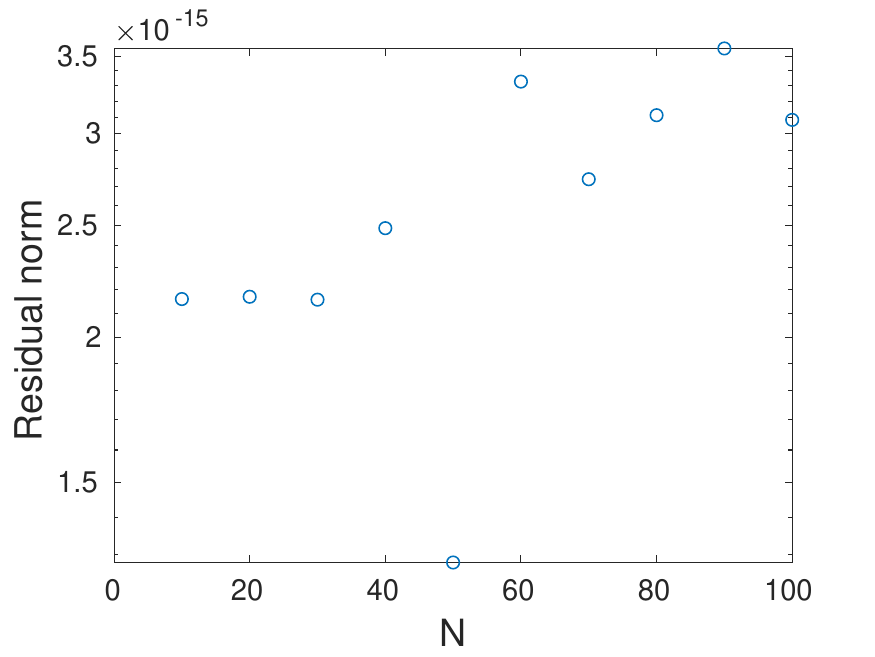} \quad
    \includegraphics[width = 0.3\textwidth]{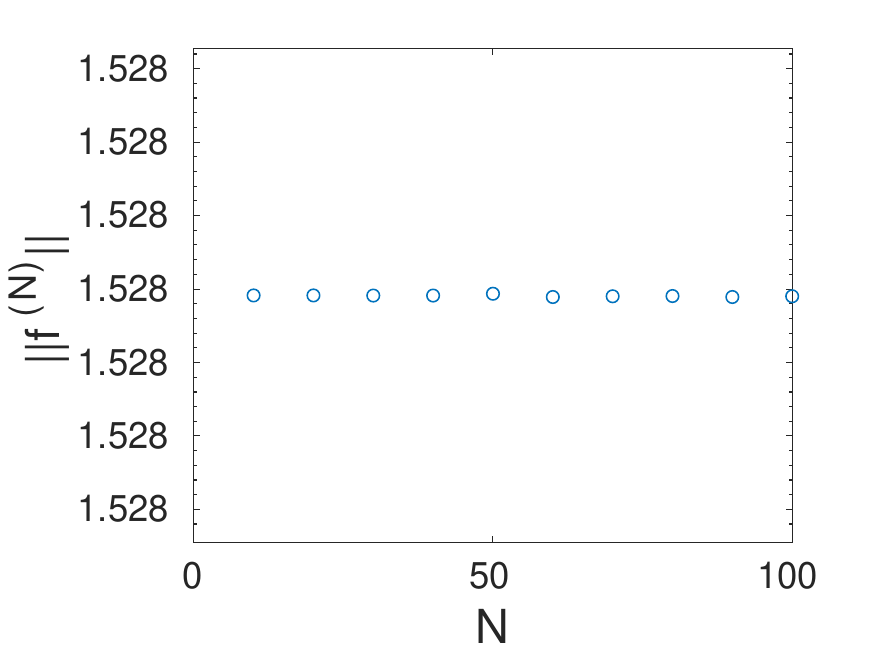}
    \caption{Legendre basis truncation}
  \end{subfigure}
  \begin{subfigure}[b]{\textwidth}
    \includegraphics[width = 0.3\textwidth]{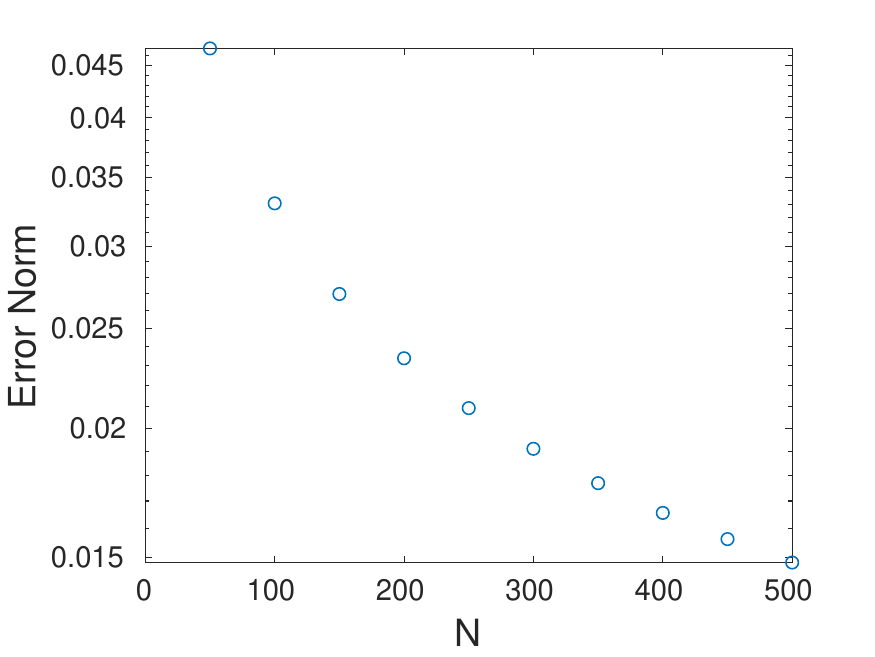} \quad
    \includegraphics[width = 0.3\textwidth]{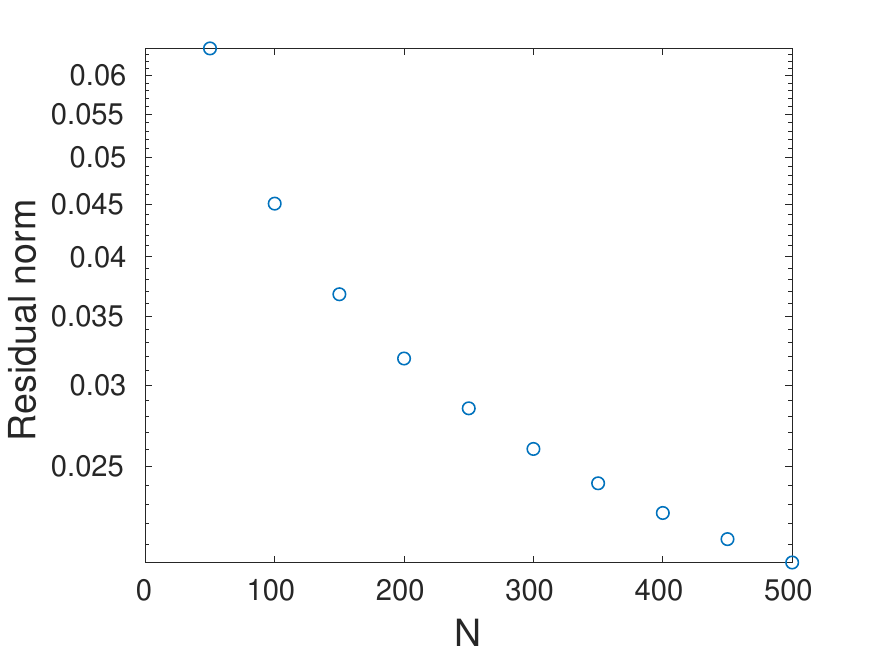} \quad
    \includegraphics[width = 0.3\textwidth]{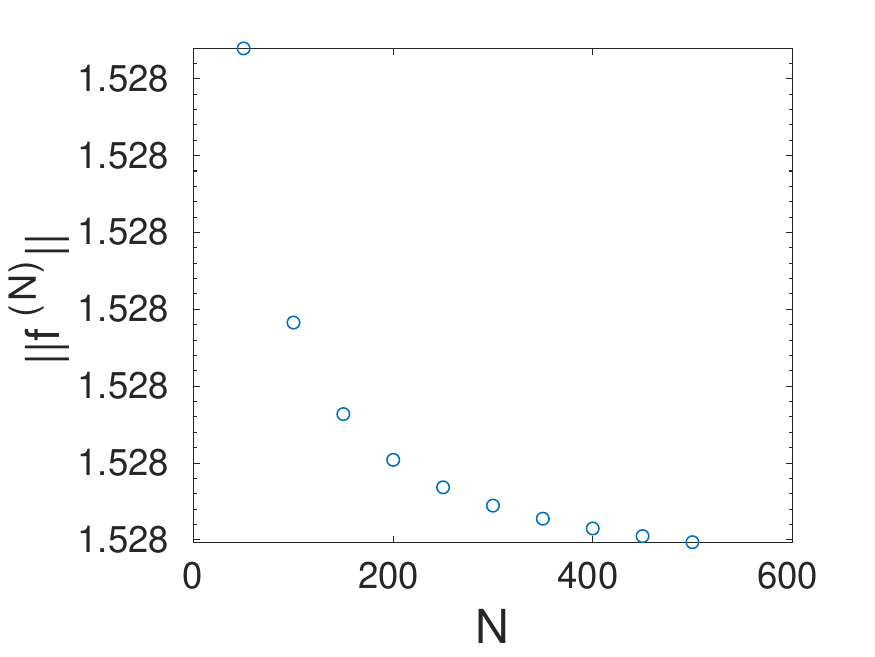}
    \caption{Complex Fourier basis truncation}
  \end{subfigure}
  \begin{subfigure}[b]{\textwidth}
    \includegraphics[width = 0.3\textwidth]{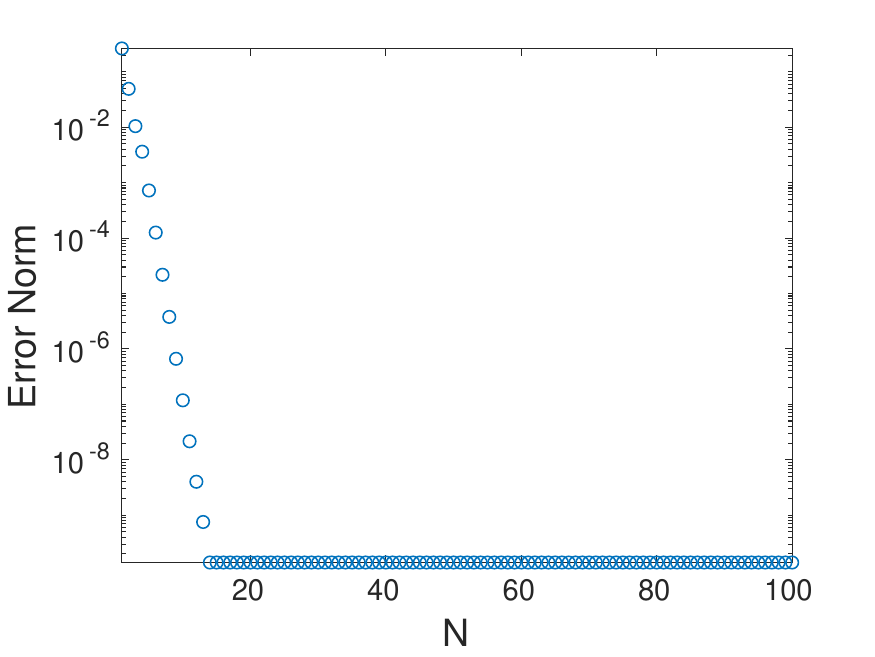} \quad
    \includegraphics[width = 0.3\textwidth]{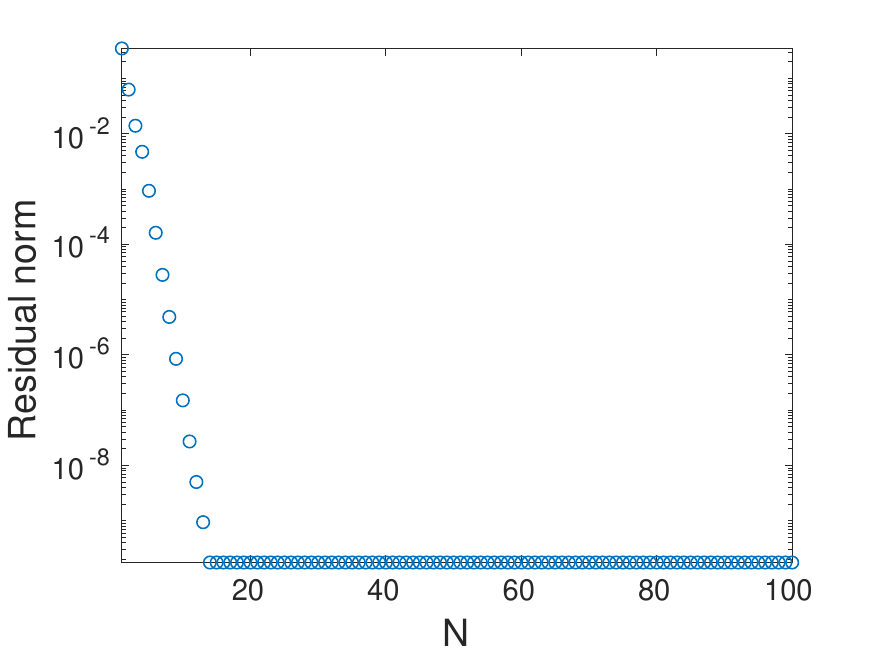} \quad
    \includegraphics[width = 0.3\textwidth]{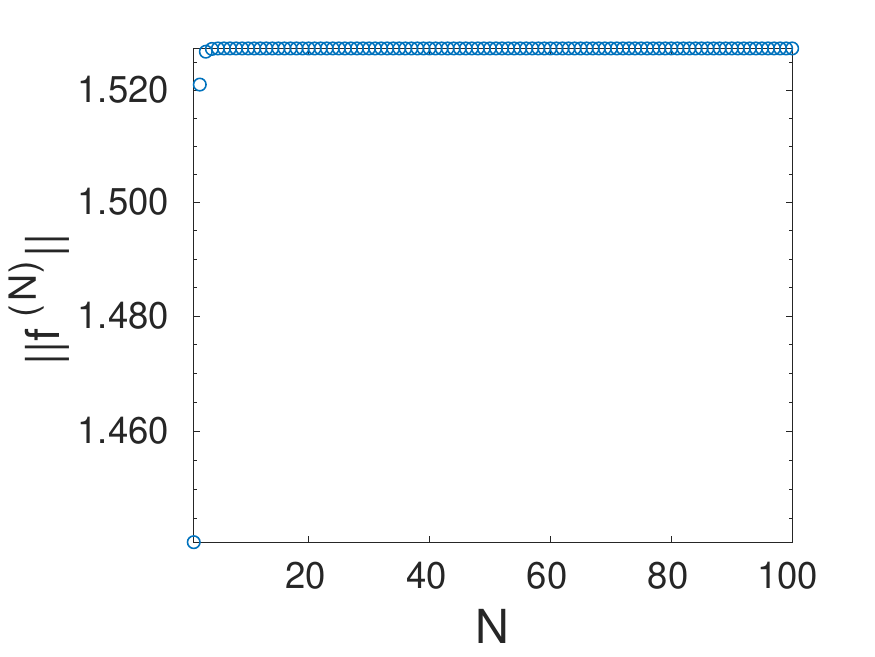} \quad
    \caption{Krylov basis truncation}
  \end{subfigure}
  \caption{Norm of the infinite-dimensional error, residual, and approximated solution for the $M$-multiplication inverse problem truncated with the Legendre, complex Fourier, and Krylov bases.} \label{fig:Mult_basis}
\end{figure}

\begin{figure}[!t]
  \centering
  \begin{subfigure}[b]{0.45\textwidth}
    \includegraphics[width = \textwidth]{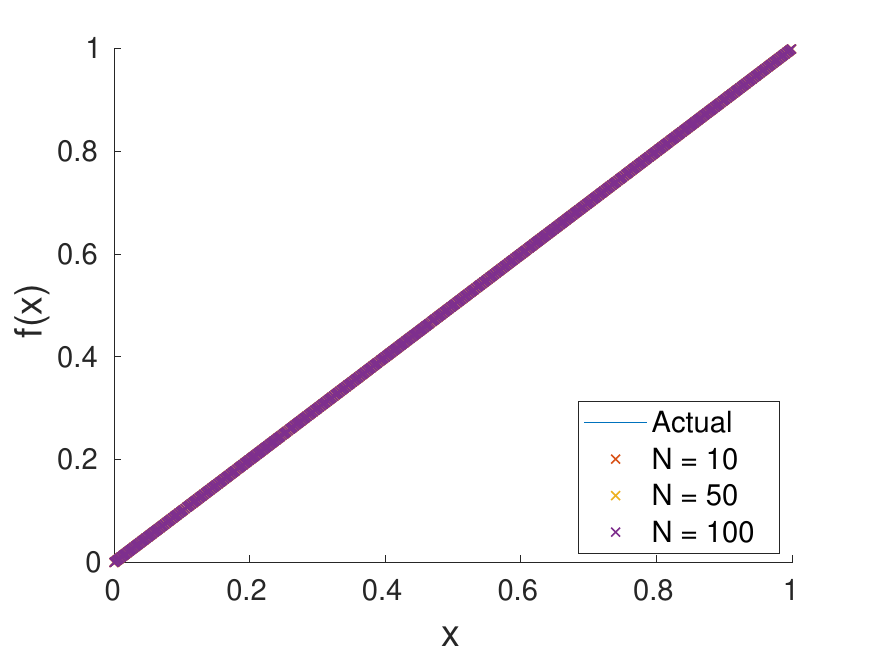}
    \caption{Legendre basis}
  \end{subfigure}
  \begin{subfigure}[b]{0.45\textwidth}
    \includegraphics[width = \textwidth]{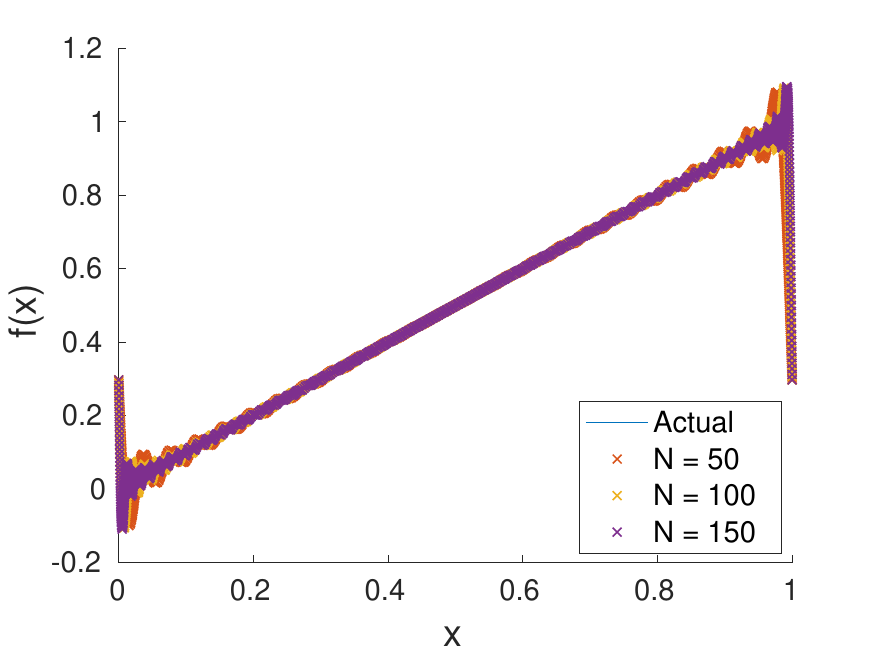}
    \caption{Complex Fourier basis}
  \end{subfigure} \\
  \begin{subfigure}[b]{0.45\textwidth}
    \includegraphics[width = \textwidth]{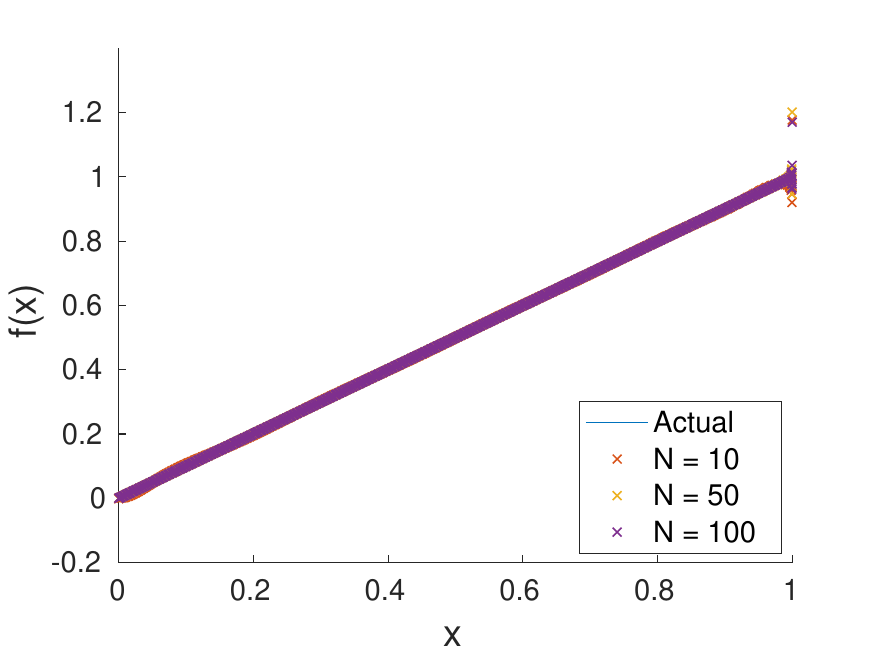}
    \caption{Krylov basis}
  \end{subfigure}   
  \caption{Reconstruction of the exact solution $f_1(x) = x$ from the solutions for the problem $Vf_1=g_1$. The Fourier basis produces an inaccurate reconstruction due to high oscillations, resulting in higher errors.} \label{fig:Volt_recon}
\end{figure}

\begin{figure}[!t]
  \centering
  \begin{subfigure}[b]{0.45\textwidth}
    \includegraphics[width = \textwidth]{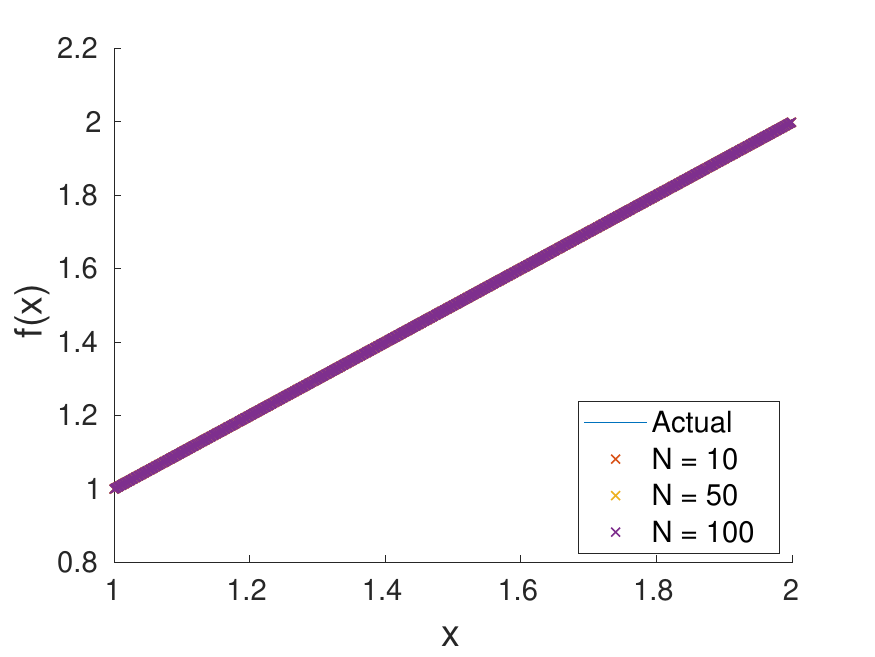}
    \caption{Legendre basis}
  \end{subfigure}
  \begin{subfigure}[b]{0.45\textwidth}
    \includegraphics[width = \textwidth]{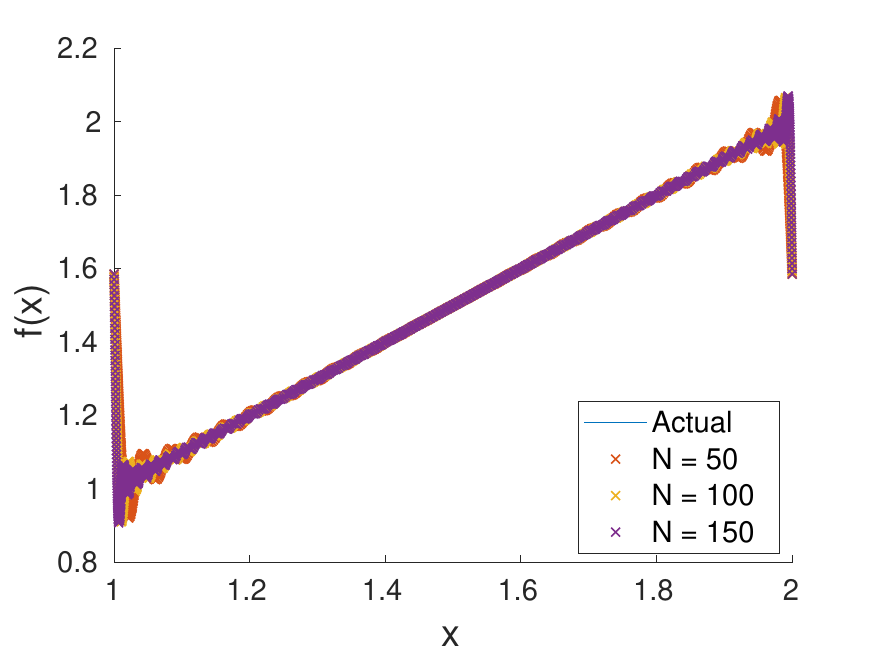}
    \caption{Complex Fourier basis}
  \end{subfigure}  
  \begin{subfigure}[b]{0.45\textwidth}
    \includegraphics[width = \textwidth]{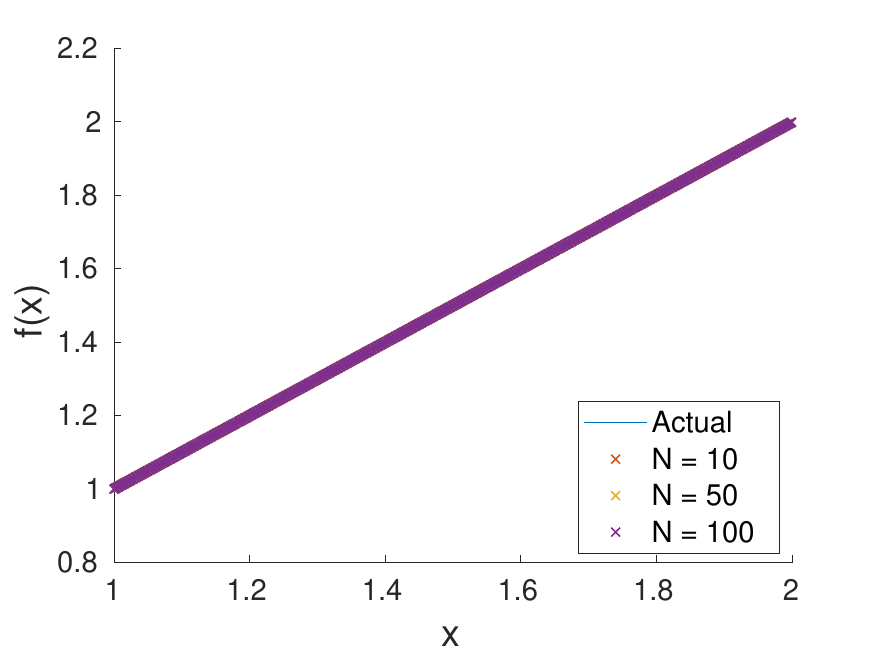}
    \caption{Krylov basis}
  \end{subfigure}  
  \caption{Reconstruction of the exact solution $f_2(x) = x$ from the solutions for the problem $Mf_2=g_2$.} \label{fig:Mult_recon}
\end{figure}

We treated both problems with three different orthonormal bases: the Legendre polynomials and the complex Fourier modes (on the intervals $[0,1]$ or $[1,2]$, depending on the problem) solved using the QR factorisation algorithm, and the Krylov basis generated using the GMRES algorithm.

Computationally speaking, generating accurate representations of the Legendre polynomials is quite demanding and accuracy can be lost rather soon due to their highly oscillatory nature, particularly at the end points. For this reason we limited our investigation up to $N = 100$ when considering the Legendre basis, but $N = 500$ when considering the complex Fourier basis.  It is expected that there is no significant numerical error from the computation of the Legendre basis, as the $L^2[0, 1]$ and $L^2[1, 2]$ norms of the basis polynomials have less than 1$\%$ error compared to their exact unit value.

For each problem and each choice of the basis, we monitored the norm of the infinite-dimensional error $\|\mathscr{E}_N\|_{L^2}=\|f-\widehat{f^{(N)}}\|_{L^2}$ ($f=f_1$ or $f_2$), of the infinite-dimensional residual $\|\mathfrak{R}_N\|_{L^2}=\|g-A\,\widehat{f^{(N)}}\|_{L^2}$ ($g=g_1$ or $g_2$; $A=V$ or $M$), and of the approximated solution $\|\widehat{f^{(N)}}\|_{L^2}=\|f^{(N)}\|_{\mathbb{C}^N}$.

Figures \ref{fig:Volt_basis} and \ref{fig:Volt_recon} highlight the difference between the computation in the three bases for the Volterra operator. 
\begin{itemize}
 \item In the Legendre basis, $\|\mathscr{E}_N\|_{L^2}$ and $\|\mathfrak{R}_N\|_{L^2}$ are almost zero. $\|\widehat{f^{(N)}}\|_{L^2}$ stays bounded and constant with $N$ and matches the expected value \eqref{eq:solV}. The approximated solutions reconstruct the exact solution $f_1$  at any truncation number.
 \item In the complex Fourier basis, both $\|\mathscr{E}_N\|_{L^2}$ and $\|\mathfrak{R}_N\|_{L^2}$ are some orders of magnitude \emph{larger} than in the Legendre basis and decrease monotonically with $N$; in fact, $\|\mathscr{E}_N\|_{L^2}$ and $\|\mathfrak{R}_N\|_{L^2}$ display an evident convergence to zero, however attaining values that are more than ten orders of magnitude larger than the corresponding error and residual norms for the same $N$ in the Legendre case. $\|\widehat{f^{(N)}}\|_{L^2}$, on the other hand, increases monotonically and appears to approach the theoretical value \eqref{eq:solV}. These quite stringent differences in the error and residual may be attributable to the Gibbs phenomenon. In fact, reconstructing $f_1$ using the Krylov approximated solutions produces a vector that shows a highly oscillatory behaviour near the end points, confirming the presence of the Gibbs phenomenon. 
 \item In the Krylov basis $\|\mathscr{E}_N\|_{L^2}$ and $\|\mathfrak{R}_N\|_{L^2}$ decrease monotonically, relatively fast for small $N$'s, then rather slowly with $N$. Such quantities are smaller than in the Fourier basis. $\|\widehat{f^{(N)}}\|_{L^2}$ displays some initial highly oscillatory behaviour, but quickly approaches the theoretical value \eqref{eq:solV}. On the other hand, the reconstruction appears to be quite good with some noticeable oscillations at the end points.
 \end{itemize}

Thus, among the considered truncations the Legendre basis yields the most accurate reconstruction and the complex Fourier basis yields the least accurate reconstruction of the exact solution.


In contrast, Figures \ref{fig:Mult_basis} and \ref{fig:Mult_recon} highlight the difference between the computation in the three bases for the $M$-multiplication operator. 
\begin{itemize}
 \item In the Legendre basis, $\|\mathscr{E}_N\|_{L^2}$ and $\|\mathfrak{R}_N\|_{L^2}$ are again almost zero. $\|\widehat{f^{(N)}}\|_{L^2}$ is constant with $N$ at the expected value \eqref{eq:solM}. The approximated solutions reconstruct the exact solution $f_2$ at any truncation number.
 \item In the Fourier basis the behaviour of the above indicators is again qualitatively the same, and again with a much milder convergence rate in $N$ to the asymptotic values as compared with the Legendre case. $\|\mathscr{E}_N\|_{L^2}$ and $\|\mathfrak{R}_N\|_{L^2}$ still display an evident convergence to zero. Again the higher error compared to the Legendre case is likely due to the nature of the approximation of the exact solution $f_2$ by oscillatory functions and the Gibbs phenomenon.
 \item The Krylov basis displays a fast initial decrease of both $\|\mathscr{E}_N\|_{L^2}$ and $\|\mathfrak{R}_N\|_{L^2}$ to the tolerance level of $10^{-10}$ that was set for the residual. $\|\widehat{f^{(N)}}\|_{L^2}$ also increases rapidly and remains constant at the expected value \eqref{eq:solM}. The reconstruction of the solution is excellent, but still not quite as good as the Legendre case.
\end{itemize}


All this gives numerical evidence that the choice of the truncation basis \emph{does} affect the sequence of solutions. The Legendre basis is best suited to these problems as $f_1$, $f_2$, $g_1$ and $g_2$ are perfectly representable by the first few basis vectors.

\def\cprime{$'$}

\end{document}